\documentclass[11pt]{article}

\usepackage{amsmath,amssymb,amsthm,amscd,epsfig}
\newcommand{\vp}{\varepsilon}

\newcommand{\ul}{\underline}

\theoremstyle{plain}

\newtheorem{thm}{Theorem}

\newtheorem{lem}{Lemma}
\newtheorem{cor}{Corollary}

\theoremstyle{definition}

\newtheorem{defn}{Definition}

\theoremstyle{remark}

\begin{document}

\title{Thermodynamic formalism for invariant measures in iterated function systems with overlaps}
\author{Eugen Mihailescu}

\date{}
\maketitle

\large{Preprint of an article published in \ Communications in Contemporary Mathematics, 2021, DOI: 10.1142/S0219199721500413.}

\

\begin{abstract}
We study images of equilibrium (Gibbs) states for a class of non-invertible transformations associated to conformal iterated function systems with overlaps $\mathcal S$. We prove exact dimensionality for these image measures, and find a dimension formula  using their  overlap numbers. 

In particular, we obtain a geometric formula for the dimension of self-conformal measures for  iterated function systems with  overlaps, in terms of the overlap numbers.  
This implies a necessary and sufficient condition for dimension drop. If $\nu = \pi_*\mu$ is a self-conformal measure, then $HD(\nu) < \frac{h(\mu)}{|\chi(\mu)|}$ if and only if the overlap number $o(\mathcal S, \mu) > 1$.  Examples are also discussed.
\end{abstract}

\

\textbf{Mathematics Subject Classification 2010:} 37A35, 37D20, 37C45, 28A80.

\textbf{Keywords:} Thermodynamic formalism;  Hausdorff dimension; Lyapunov exponents; conformal iterated function systems with overlaps; hyperbolic endomorphisms; overlap numbers for measures.

\section{Introduction and main results.}

In this paper we study the thermodynamic formalism and dimension for images of equilibrium measures, for noninvertible transformations associated to conformal iterated function systems with overlaps $\mathcal S$. We prove  \textbf{exact dimensionality} for this new class of image measures; this implies that all the dimensions of these measures (Hausdorff, pointwise, box) coincide. A \textbf{formula for the dimension} of these measures is obtained, in terms of their entropy, Lyapunov exponent and  \textbf{overlap number}  (which represents the average rate of growth of the number of generic self-intersections in the  limit set $\Lambda$ of $\mathcal S$).

In particular, we prove the exact dimensionality of  arbitrary \textbf{self-conformal measures} for conformal iterated function systems with overlaps, and we determine a  \textbf{dimension formula} for self-conformal measures in terms of overlap numbers, entropy and Lyapunov exponents.

When the iterated function system (IFS) satisfies Open Set Condition, the dimension of the self-conformal measure $\nu:= \pi_*\mu$ is equal to the value ``entropy divided by Lyapunov exponent'' (see \cite{F1}). 
For \textbf{ IFS with overlaps} we show that for a  self-conformal measure $\nu_{\bf p} = \pi_*\mu_{\bf p}$, \  $HD(\nu_{\bf p}) = \frac{h(\mu_{\bf p})}{|\chi(\mu_{\bf p})|}$ if and only if the overlap number of $\mu_{\bf p}$ is equal to 1, and we say in that case that the system $\mathcal S$ is \textbf{separated $\mu_{\bf p}$-asymptotically}. If $\nu_0$ is the equally distributed self-conformal measure, then $HD(\nu_0) = \frac{h(\mu_0)}{|\chi(\mu_0)|}$ if and only if the topological overlap number of $\mathcal S$ is equal to 1.

Thus for self-conformal measures, our results establish a necessary and sufficient condition for the   \textbf{dimension drop} of $\pi_*\mu$ from the value $\frac{h(\mu)}{|\chi(\mu)|}$, namely that the overlap number $o(\mathcal S, \mu) >1$.

The exact dimensionality of self-conformal measures (and other measures) on limit sets of finite conformal IFS with overlaps was proved by Feng and Hu in the groundbreaking paper \cite{FH}, and a dimension formula was obtained by them in terms of  the  entropy,  Lyapunov exponent and projection entropy. 
Our proof of exact dimensionality is different, and it follows from our general result for hyperbolic endomorphisms. The dimension formula that we obtain is partially geometric, and we apply it  to certain examples.

\

  In the sequel, we study the analytic and stochastic properties for equilibrium measures over  lift spaces associated to  conformal iterated systems with overlaps, and for a class of their push-forward measures. In the process, we investigate also the intricate interlacing  in a typical trajectory in $\Lambda$ of  generic iterates and the gaps of  non-generic iterates, and how they influence the local densities of the  measures. 
The general setting  is the following: 

Let $\mathcal S = \{\phi_i, i \in I\}$ be an arbitrary finite iterated function system of smooth conformal injective contractions of a compact set with nonempty interior $V \subset \mathbb R^D, D \ge 1$. We do not assume any kind of separation condition for $\mathcal S$ \ (see for eg  \cite{F}, \cite{Hu}, \cite{LNW} for some possible separation conditions). The \textit{limit set} of the system $\mathcal S$ is given by: $$\Lambda = \mathop{\cup}_{\omega \in \Sigma_I^+}\mathop{\cap}\limits_{n \ge 0} \phi_{\omega_1\ldots \omega_n}(V),$$
where $\Sigma_I^+$ is the 1-sided symbolic space on $|I|$ symbols, and $\omega = (\omega_1, \ldots, \omega_n, \ldots) \in \Sigma_I^+$ is arbitrary (see for eg \cite{F}, \cite{Hu}).  
 Denote by $[\omega_1\ldots\omega_n]$ the cylinder on the first $n$ elements of $\omega$, and by $\phi_{i_1\ldots i_p}:= \phi_{i_1}\circ \ldots \circ \phi_{i_p}$. The shift $\sigma: \Sigma_I^+ \to \Sigma_I^+$ is given by $\sigma(\omega) = (\omega_2, \omega_3, \ldots),  \omega \in \Sigma_I^+$. We endow the space $\Sigma_I^+ \times \Lambda$ with the product metric. Also let the canonical coding map, $$\pi: \Sigma_I^+ \to \Lambda, \ \pi(\omega) := \phi_{\omega_1\omega_2\ldots}(V)$$ 
 
Now consider the  non-invertible skew product transformation  on the metric space $\Sigma_I^+\times \Lambda$, 
$$
\Phi: \Sigma_I^+ \times \Lambda \to \Sigma_I^+ \times \Lambda, \  \ \Phi(\omega, x) = (\sigma \omega, \phi_{\omega_1}(x)), \ (\omega, x) \in \Sigma_I^+ \times \Lambda$$
The endomorphism $\Phi$ has  a type of hyperbolic structure, since it is expanding in the first coordinate and contracting in the second coordinate (due to the uniform contractions in $\mathcal S$).
The \textit{pressure functional} of $\sigma$ is defined for general continuous potentials $g$ on $\Sigma_I^+$ as $P_\sigma: \mathcal C(\Sigma_I^+) \to \mathbb R$ (for eg \cite{Bo}, \cite{KH}, \cite{Ru-78}).
Consider now a H\"older continuous potential $\psi: \Sigma_I^+ \to \mathbb R$, and let the functional $F_\psi$ defined on the space $\mathcal M(\sigma)$ of $\sigma$-invariant probability measures on $\Sigma_I^+$  by $$F_{\psi}: \mathcal M(\sigma) \to \mathbb R, \ F_\psi(\mu) := h_\sigma(\mu) + \int_{\Sigma_I^+} \psi \ d\mu,$$
where $h_\sigma(\mu)$ is the measure-theoretic entropy of $\mu$.
Then the supremum of $F_\psi$ is equal to the pressure $P_\sigma(\psi)$ of $\psi$, and is attained at a unique measure, called the \textit{equilibrium measure} of $\psi$ and denoted by $\mu_\psi$.  
Since $\psi$ was assumed  to be H\"older continuous, the notion of equilibrium measure is equivalent to that of Gibbs measure (for eg \cite{Bo}, \cite{KH}, \cite{Ru-78}). 

Let $\pi_1: \Sigma_I^+\times \Lambda \to \Sigma_I^+$ be the projection on the first coordinate.
Define also the potential $\hat\psi:= \psi \circ \pi_1: \Sigma_I^+ \times \Lambda \to \mathbb R$, which is H\"older continuous. Let the functional $F_{{\hat\psi}}$ defined on the space  $\mathcal M(\Phi)$ of $\Phi$-invariant probability measures on $\Sigma_I^+ \times \Lambda$, be given by $$F_{\hat\psi}: \mathcal M(\Phi) \to \mathbb R, \ F_{\hat\psi}(\mu) = h_\Phi(\mu) + \int_{\Sigma_I^+\times \Lambda} \hat\psi \ d\mu,$$ where $h_\Phi(\mu)$ is the measure-theoretic entropy of $\mu \in \mathcal M(\Phi)$ with respect to $\Phi$. Then as $\Phi$ has a hyperbolic structure, it can be shown similarly as in \cite{KH} that $F_{\hat\psi}$ attains its supremum $P_\Phi(\hat\psi)$ at a unique measure on $\Sigma_I^+ \times \Lambda$, called the \textit{equilibrium measure} of $\hat\psi$, denoted by $\mu_{\hat\psi}$ or by $\hat\mu_\psi$. In this case, $\hat\mu_\psi$ is  a Gibbs measure for $\hat \psi$ with respect to $\Phi$ on $\Sigma_I^+\times \Lambda$ (\cite{KH}, \cite{Ru-78}). 
Notice that, $\pi_{1*}\hat\mu_{\psi} = \mu_\psi$.
If $\psi$ is fixed, denote also  $\mu_\psi$ by $\mu^+$ and $\hat\mu_{\psi}$ by $\hat\mu$. 
 The projection to the second coordinate is: 
$$\pi_2: \Sigma_I^+\times \Lambda \to \Lambda, \ \pi_2(\omega, x) = x$$ 

\textbf{ The main focus} of this paper are the metric properties of the  measures $\hat\mu_\psi$ and  $\pi_{2*}\hat\mu_{\psi}$.
Let us denote by, 
\begin{equation}\label{mu1}
\nu_{1, \psi}:= (\pi\circ \pi_1)_*\hat\mu_\psi, \ \ \text{and} \ \ 
\nu_{2, \psi}:= \pi_{2*}\hat\mu_\psi
\end{equation}

Since for any $n \ge 1$, the map $\Phi^n(\omega, x) = (\sigma^n\omega, \phi_{\omega_n\ldots \omega_1}(x))$ reverses the order of $\omega_1, \ldots, \omega_n$ in its second coordinate and since $\hat\mu_\psi$ is $\Phi^n$-invariant, we call 
\begin{equation}\label{push}
\nu_{2, \psi}=\pi_{2, *}\hat\mu_\psi,
\end{equation}
 an \textbf{order-reversing projection measure}.  In general the measure $\nu_{1, \psi}$ is different from $\nu_{2, \psi}$.

Some important notions in Dimension Theory are those of lower/upper pointwise dimensions of a measure, and the notion of exact dimensional measures (see \cite{Pe}). In general, for a probability Borel measure $\mu$ on a metric space $X$,  the \textit{lower pointwise dimension} of $\mu$ at $x\in X$ is: $$\underline\delta(\mu)(x):=\mathop{\liminf}\limits_{r \to 0} \frac{\log\mu(B(x, r))}{\log r},$$  and the \textit{upper pointwise dimension} of $\mu$ at $x\in X$ is defined as: $$\overline{\delta}(\mu)(x):= \mathop{\limsup}\limits_{r\to 0} \frac{\log\mu(B(x, r))}{\log r}$$ If $\underline\delta(\mu)(x) = \overline\delta(\mu)(x)$ then we call the common value the \textit{pointwise dimension} of $\mu$ at $x$, denoted by $\delta(\mu)(x)$. If for $\mu$-a.e $x \in X$, the pointwise dimension $\delta(\mu)(x)$ exists and is constant, we say that $\mu$ is \textit{exact dimensional}. In this case there is a value $\alpha\in \mathbb R$ s.t for $\mu$-a.e $x\in X$, $$\delta(\mu)(x):=\underline\delta(\mu)(x) = \overline\delta(\mu)(x) = \alpha$$ 

In \cite{FH}, Feng and Hu defined the projection entropy for a $\sigma$-invariant probability measure $\mu$ on $\Sigma_I^+$, namely
$h_\pi(\sigma, \mu) := H_\mu(\mathcal P|\sigma^{-1}\pi^{-1}\gamma) - H_\mu(\mathcal P|\pi^{-1}\gamma),$ 
where $\pi:\Sigma_m^+\to \Lambda$ is the canonical coding map, $\mathcal P$ is the partition with 0-cylinders $\{[i], i \in I\}$ of $\Sigma_I^+$, and $\gamma$ is the $\sigma$-algebra of Borel sets in $\mathbb R^d$. It was shown in \cite{FH} that if $\mu$ is ergodic then for $\mu$-a.e $\omega \in \Sigma_I^+$,  
$$
\delta(\pi_*\mu)(\pi\omega) = \frac{h_\pi(\sigma, \mu)}{-\int_{\Sigma_I^+}\log |\phi_{\omega_1}'(\pi\sigma\omega)| \ d\mu(\omega)},
$$
hence $\pi_*\mu$ is exact dimensional. This is equivalent, in our notation, to the fact that $\nu_1$ is exact dimensional. Our approach and methods in the sequel are however different, as we study the order-reversing image measure $\pi_{2*}\mu$ and the measure $\hat\mu$.
 
 Denote the \textit{stable Lyapunov exponent} of  $\Phi$ with respect to the measure $\hat\mu$ on $\Sigma_I^+\times \Lambda$ by, 
 \begin{equation}\label{sle}
 \chi_s(\hat\mu):= \int_{\Sigma_I^+\times\Lambda}\log|\phi'_{\omega_1}(x)| \ d\hat\mu(\omega, x)
 \end{equation}
For a shift-invariant measure $\mu$ on $\Sigma_I^+$, denote  the \textit{Lyapunov exponent} of $\mu$ with respect to $\mathcal S$ by,
\begin{equation}\label{Lex}
\chi(\mu):= \int_{\Sigma_I^+}\log|\phi'_{\omega_1}(\pi\sigma\omega)| \ d\mu(\omega)
\end{equation}

We will use the Jacobian in the sense of Parry \cite{Pa}; consider the Jacobian $J_\Phi(\hat\mu)$ of a $\Phi$-invariant measure $\hat \mu$ on $\Sigma_I^+ \times \Lambda$. Then $J_\Phi(\hat \mu) \ge 1$ for $\hat\mu$-a.e $(\omega, x) \in \Sigma_I^+\times \Lambda$, and for $\hat\mu$-a.e $(\omega, x) \in \Sigma_I^+ \times \Lambda$, $$J_\Phi(\hat\mu)(\omega, x) = \mathop{\lim}\limits_{r\to 0} \frac{\hat\mu(\Phi(B((\omega, x), r)))}{\hat\mu(B((\omega, x), r))}$$
From the Chain Rule for Jacobians,  
$J_{\Phi^n}(\hat\mu)(\omega, x) = J_\Phi(\hat \mu)(\Phi^{n-1}(\omega, x))\cdot \ldots \cdot J_\Phi(\hat\mu)(\omega, x)$ for  $n \ge 1$,
and from Birkhoff Ergodic Theorem applied to $\log J_\Phi(\hat\mu)(\cdot, \cdot)$, we have that for $\hat\mu$-a.e $(\omega, x) \in \Sigma_I^+\times \Lambda$,
\begin{equation}\label{BET}
\frac{\log J_{\Phi^n}(\hat \mu)(\omega, x)}{n} \mathop{\longrightarrow}\limits_{n \to \infty} \int_{\Sigma_I^+\times \Lambda} \log J_{\Phi}(\hat\mu)(\eta, y) \ d\hat\mu(\eta, y)
\end{equation}
Ruelle introduced in \cite{Ru-fold}, \cite{Ru-survey}, the notion of \textit{folding entropy} $F_f(\nu)$ of a measure $\nu$ invariant with respect to an endomorphism $f: X \to X$ on a Lebesgue space $X$, as being the conditional entropy $H_\nu(\epsilon|f^{-1}\epsilon)$, where $\epsilon$ is the point partition of $X$ and $f^{-1}\epsilon$ is the fiber partition. In fact from \cite{Pa}, \cite{Ru-fold},  
$F_f(\nu) = \int_X \log J_f(\nu) d\nu$.
Thus for the $\Phi$-invariant measure $\hat \mu$ on $\Sigma_I^+\times \Lambda$, 
\begin{equation}\label{FlogJ}
F_\Phi(\hat \mu)  = \int_{\Sigma_I^+\times \Lambda} \log J_\Phi(\hat \mu) \ d\hat \mu
\end{equation}

In our case, the folding entropy turns out to be related to the overlap number of $\hat\mu$. 
The notion of \textit{overlap number} $o(\mathcal S, \mu_g)$ for an equilibrium measure $\mu_g$ of a H\"older continuous potential $g:\Sigma_I^+ \times \Lambda \to \mathbb R$ was introduced  in \cite{MU-JSP2016}, and represents an average asymptotic rate of growth for the number of generic overlaps of order $n$ in $\Lambda$. Namely, for any $\tau>0$, let the set of generic preimages with respect to $\mu_g$ having the same $n$-iterates as $(\omega, x)$, 
$$\Delta_n((\omega, x), \tau, \mu_g) := \{(\eta_1, \ldots, \eta_n) \in I^n, \exists y \in \Lambda, \phi_{\omega_n\ldots \omega_1}(x) = \phi_{\eta_n\ldots \eta_1}(y),  \ |\frac{S_ng(\eta, y)}{n} - \int_{\Sigma_I^+ \times \Lambda} g\ d\mu_\psi| < \tau\},$$
where $(\omega, x) \in \Sigma_I^+ \times \Lambda$ and $S_n g(\eta, y)$ is the consecutive sum of $g$ with respect to $\Phi$. Denote  by $$b_n((\omega, x), \tau, \mu_g):= Card \Delta_n((\omega, x), \tau, \mu_g)$$
Then, in \cite{MU-JSP2016} we showed that the following limit exists and defines the \textit{overlap number} of $\mu_g$,
$$o(\mathcal S, \mu_g) = \exp\big(\mathop{\lim}\limits_{\tau \to 0} \mathop{\lim}\limits_{n \to \infty} \frac 1n \int_{\Sigma_I^+ \times \Lambda} \log b_n((\omega, x), \tau, \mu_g) \ d\mu_g(\omega, x) \big)$$ 
Clearly $o(\mathcal S, \mu_g) \ge 1$.
There is also a relation between overlap number and folding entropy, 
\begin{equation}\label{oS}
o(\mathcal S, \mu_g) = \exp(F_\Phi(\mu_g))
\end{equation} 

If $\mu_g = \hat\mu_0$ is the measure of maximal entropy of $\Phi$ on $\Sigma_I^+\times \Lambda$, denote $o(\mathcal S, \hat\mu_0)$ by $o(\mathcal S)$ and call it \textit{the topological overlap number} of $\mathcal S$. All preimages are generic in this case. For $n \ge 1$, $x \in \Lambda$, let
 \begin{equation}\label{beta}
 \beta_n(x) := Card\{(j_1, \ldots, j_n) \in  I^n, \ x \in \phi_{j_1}\circ\ldots \circ \phi_{j_n}(\Lambda)\}
 \end{equation}
If $\mu_0$ denotes the measure of maximal entropy on $\Sigma_I^+$, then the topological overlap number satisfies:
\begin{equation}\label{topo}
o(\mathcal S) = \exp\big(\mathop{\lim}\limits_{n \to \infty} \frac 1n \int_{\Sigma_I^+} \log \beta_n(\pi\omega) \ d\mu_0(\omega) \big),
\end{equation}
hence $o(\mathcal S)$ is an average rate of growth of the number of intersections between sets of type $\phi_{i_1\ldots i_n}(\Lambda)$,  $i_1, \ldots, i_n \in I$ and $n \to \infty$.

 \

 Dynamics and dimension for  dynamical systems with some form of hyperbolicity attracted a lot of interest and were studied for eg in \cite{BPS}, \cite{Bo}, \cite{ER}, \cite{LY}, \cite{Ma},  \cite{Pe}, \cite{PW}, \cite{Ru-78}, \cite{Ru-survey}, \cite{Y}, to mention a few.
Also endomorphisms (non-invertible maps) were studied for example in \cite{Mane}, \cite{Ma1} -  \cite{MU-BLMS}, \cite{Pa},  \cite{PW}, \cite{Ru-fold} -  \cite{ST}. 
The problem of dimension in conformal iterated function systems with or without overlaps was studied in  \cite{BF}, \cite{DN}, \cite{FJ} - \cite{Hu}, \cite{Ke} - \cite{LNW}, \cite{MU-JSP2016},  \cite{MU-Adv},  \cite{O}, \cite{PeS},  \cite{PS},  \cite{Sh} - \cite{ShSo} to mention a few. Exact dimensionality and dimension formulas for invariant measures were also intensely studied over the years.
In \cite{Ma} Manning showed that for an Axiom A diffeomorphism of a surface
preserving an ergodic measure $\mu$, the entropy $h(\mu)$ is equal to the product of the positive
Lyapunov exponent of $\mu$ and the dimension of the set of $\mu$-generic points in an unstable
manifold. 
 In  \cite{Y} Young proved that the Hausdorff dimension of a hyperbolic invariant measure $\mu$ for a surface diffeomorphism is given by the entropy and the Lyapunov exponents,  $HD(\mu) = h(\mu)(\frac 1{\chi_u(\mu)} - \frac 1{\chi_s(\mu)})$. In \cite{LY} Ledrappier and Young proved a formula for the entropy of an invariant
measure $\mu$ for a diffeomorphism of a compact Riemannian manifold, in terms of Lyapunov
exponents and dimensions of $\mu$ in the respective stable/unstable directions. In \cite{Ma1} Manning studied  the dimension for the maximal measure of a polynomial map. And in \cite{Mane} Ma\~ne proved exact dimensionality for ergodic measures invariant to rational maps.   In \cite{PW} Pesin and Weiss verified the
Eckmann-Ruelle Conjecture (\cite{ER}) for equilibrium measures for H\"older continuous conformal
expanding maps and conformal Axiom A (topologically hyperbolic) homeomorphims; and constructed an Axiom A homeomorphism
for which the  measure of maximal entropy
 has different upper and lower pointwise dimensions a.e, so in this case the Eckmann-Ruelle Conjecture is false.
Then, in \cite{BPS} Barreira, Pesin and Schmeling showed that every hyperbolic measure $\mu$ invariant under a $C^{1+\vp}$  diffeomorphism
of a smooth Riemannian manifold has asymptotically almost local
product structure and  proved the Eckmann-Ruelle conjecture, namely the pointwise
dimension of $\mu$ exists almost everywhere, thus $\mu$ is exact dimensional. In \cite{PeS} Peres and Solomyak showed the existence of $L^q$-dimensions and entropy dimension for  self-conformal measures. Feng and Hu proved in \cite{FH} that the canonical projection of any ergodic measure from the shift space for a finite conformal iterated function system with overlaps, is exact dimensional on the limit set,  and found the Hausdorff dimension of this projection measure by using a notion of projection entropy. In \cite{FJ} Falconer and Jin proved  that the random multiplicative cascade measures  on self-similar sets and their projections and sections are almost surely exact dimensional.  
For a class of hyperbolic endomorphisms it was shown in \cite{M-ETDS11} that the conditional measures of equilibrium measures on the stable manifolds are geometric, and thus exact dimensional. 
For random countable iterated function systems with arbitrary overlaps, Mihailescu and Urba\'nski showed in \cite{MU-Adv} that the projection of any ergodic measure from the shift space which satisfies a finite entropy condition, is exact dimensional, and found a formula for its dimension and gave applications. 
In \cite{BK} Barany and K\"aenm\"aki studied some self-affine measures.

Our current result is different in the sense that it treats general conformal iterated function systems with overlaps and a class of invariant measures, including self-conformal measures,  by relating the dimension of  measures with their overlap numbers. Our formula has a geometric character, and the proof uses different methods, coming from dynamics of endomorphisms. Also, we obtain a necessary and sufficient condition for dimension drop from the value ``entropy divided by Lyapunov exponent''. Related to the problem of dimension drop Hochman \cite{Ho} showed for self-similar measures on $\mathbb R$ that if the dimension is strictly smaller than the similarity dimension and 1, then there is a super-exponential concentration of cylinders. Our results show that  an arbitrary self-conformal measure $\nu$ in $\mathbb R^D, D \ge 1$ has a dimension drop if and only if, the overlap number of that measure is strictly larger than 1. We introduce also the notion of $\mu_{\bf p}$-asymptotically separated systems. 
The formula we obtain can be used for dimension estimates in  non-linear examples, including for instance mixed Julia sets. 


\

Our \textbf{main results}  are the following:
\newline
First, in \textbf{Theorem \ref{thm1}}  we prove the exact dimensionality and dimension formula for the general push-forward measure $\nu_{2, \psi}$.

\begin{thm}\label{thm1}
Let $\mathcal S$ be a finite conformal iterated function system on a compact set with non-empty interior $V \subset \mathbb R^D, D \ge 1$, with limit set $\Lambda$, and $\psi$ be a H\"older continuous potential on $\Sigma_I^+$ with equilibrium measure $\mu_\psi$, and let $\hat \mu_\psi$ be the equilibrium measure of $\psi \circ \pi_1$ on $\Sigma_I^+\times \Lambda$ with respect to $\Phi$. Denote   $\nu_{2, \psi}:= \pi_{2*}\hat\mu_\psi$. Then the measure $\nu_{2, \psi}$ is exact dimensional on $\Lambda$, and for $\nu_{2, \psi}$-a.e. $x\in \Lambda$,
$$HD(\nu_{2, \psi}) = \delta(\nu_{2, \psi})(x) =  \frac{ h_\sigma(\mu_\psi) - \log(o(\mathcal S, \hat \mu_\psi))}{|\chi_s(\hat\mu_\psi)|}.$$

\end{thm}

\

The proof of this Theorem  has several parts. The proof for the lower bound for dimension is  the most difficult of these parts and contains some new methods from dynamics of endomorphisms. It is based on an intricate study of the interlacing  of generic iterates with respect to $\hat \mu_\psi$ and of the maximal lengths of ``gaps'' consisting of  non-generic iterates in trajectories, and how these are  involved in computing local densities of $\nu_{2, \psi}$. We apply Borel Density Lemma on leaves of type $\phi_{i_1\ldots i_m}\Lambda$ to get measure estimates.

\

Then, we obtain applications of Theorem \ref{thm1} to dimension formulas  for several \textbf{cases}:

\ \ \textbf{1.} An important particular case is that of \textbf{self-conformal measures} for arbitrary \textbf{conformal iterated function systems with overlaps} in $\mathbb R^D, D \ge 1$. 

 Let the IFS $\mathcal S$ as above, and a probability vector $\textbf p = (p_1, \ldots, p_{|I|})$, and $\mu_{\textbf p}$ be the associated Bernoulli measure on $\Sigma_I^+$. Any Bernoulli measure $\mu_{\textbf p}$ on $\Sigma_I^+$ is the equilibrium measure of some H\"older continuous potential $\psi_{\textbf p}$. 
 Then denote by $\hat\mu_{\textbf p}$ the lift of $\mu_{\textbf p}$ to $\Sigma_I^+\times \Lambda$, which is obtained as the equilibrium measure of $\psi_{\textbf p} \circ \pi_1$.  And denote by $\nu_{1, \textbf p}, \nu_{2, \textbf p}$ the associated projected measures $\nu_1, \nu_2$. 
In this case we showed in \cite{MU-JSP2016}  that $$\nu_{1, \textbf p} = \nu_{2, \textbf p},$$ so $\delta(\nu_{1, \textbf p}) = \delta(\nu_{2, \textbf p}).$  Recall the definition of Lyapunov exponent $\chi(\mu_{\bf p})$  from (\ref{Lex}). Let us denote also the \textit{overlap number of $\mu_{\bf p}$} by, 
\begin{equation}\label{ovsc}
o(\mathcal S, \mu_{\bf p}) := o(\mathcal S, \hat \mu_{\bf p})
\end{equation}

Then the \textbf{dimension of an arbitrary self-conformal measure} $ \nu_{\bf p}$   is given by:

\begin{thm}\label{bernou}
Let $\mathcal S = \{\phi_i, 1 \le i \le m\}$ be a system of injective conformal contractions on a  compact  set with non-empty interior $V \subset \mathbb R^D, D \ge 1$, with limit set $\Lambda$, and
consider an arbitrary probability vector  $\bf p = (p_1, \ldots, p_m)$. Let also $\mu_{\bf p}$  be the Bernoulli measure on $\Sigma_m^+$ associated to $\bf p$, and $\nu_{\bf p}= \pi_*\mu_{\bf p}$ be its canonical projection  on $\Lambda$. 
Then,

$$HD(\nu_{\bf p}) =  \frac{ -\mathop{\sum}\limits_{1\le i \le m} p_i\log p_i - \log(o(\mathcal S, \mu_{\bf p}))}{|\chi(\mu_{\bf p})|}.$$
  
\end{thm}

\

From Theorem \ref{bernou}, we obtain a necessary and sufficient condition for  \textbf{dimension drop} for self-conformal measures in IFS with overlaps in $\mathbb R^D, D \ge 1$.

\begin{cor}\label{drop}
In the setting of Theorem \ref{bernou}, a self-conformal measure $\nu_{\bf p}$ satisfies: $$HD(\nu_{\bf p}) <  \frac{ h(\mu_{\bf p})}{|\chi(\mu_{\bf p})|},$$
if and only if $o(\mathcal S, \mu_{\bf p}) >1$. 
\end{cor}

\begin{defn}\label{sepasymp}
In the above setting, let a Bernoulli measure $\mu_{\bf p}$. If $o(\mathcal S, \mu_{\bf p}) = 1$, then we say that the system $\mathcal S$ is \textbf{separated $\mu_{\bf p}$-asymptotically}.  
\end{defn}
It follows from Corollary \ref{drop} that, $HD(\nu_{\bf p}) = \frac{h(\mu_{\bf p})}{\chi(\mu_{\bf p})}$ \ if and only if the system $\mathcal S$ is separated $\mu_{\bf p}$-asymptotically. 
\newline
Clearly, if $\mathcal S$ satisfies the Open Set Condition, then $\mathcal S$ is separated $\mu_{\bf p}$-asymptotically for every $\mu_{\bf p}$.

\

\ \ \textbf{2.} \ Assume now we can  \textbf{bound} the number of intersections between images of various cylinders.
These estimates apply to a class of non-linear examples,  and are \textbf{stable under perturbations}.

Let $\mathcal S = \{\phi_i, 1 \le i \le m\}$ be a system of injective conformal contractions on a  compact set with non-empty interior $V \subset \mathbb R^D, D \ge 1$, with limit set $\Lambda$. Assume there exists an integer $q \ge 1$ and an open set $W\subset V$ so that $\phi_i(W) \subset W, 1 \le i \le m$ and the collection of initial cylinders of length $q$ from $\Sigma_m^+$ (i.e cylinders $[i_1, \ldots, i_q]$ with $i_1$ on position $1, \ldots, i_q$ on position $q$)  can be partitioned into $s$ subcollections: 
\begin{equation}\label{gj}
G_1, \ldots, G_s,
\end{equation}
 so that for any $C= [j_1\ldots j_q]\in G_i$ and $C' =[j_1'\ldots j_q'] \in G_j$ with $i \ne j, 1\le i, j \le s$, we have $$\phi_{j_1\ldots j_q}(W) \cap \phi_{j_1'\ldots j_q'}(W) = \emptyset.$$ However the images of cylinders from the same  $G_i$ can intersect in any way.
Denote  by $U_j$ the union of all cylinders from $G_j$, for $1 \le j \le s$, and let 
\begin{equation}\label{mj}
m_1:= Card \ G_1, \ldots, m_s := Card \ G_s
\end{equation}

In the above notation, we obtain next a computable lower bound for the dimension of $\pi_2$-image measures, by using the Borel measurable function 
$\theta:\Sigma_m^+ \to \mathbb R$, $$\theta(\omega) = \log m_j, \ \text{for} \ \omega \in U_j,  1 \le j \le s$$

\begin{cor}\label{qint}
In the above setting,
consider  $\psi: \Sigma_m^+ \to \mathbb R$ a H\"older continuous potential, and let $\mu_\psi$ be its equilibrium measure with respect to $\sigma$ on $\Sigma_m^+$, and $\hat\mu_\psi$ be the equilibrium measure of $\psi\circ \pi_1$ with respect to $\Phi$ on $\Sigma_m^+ \times \Lambda$; recall also the above notation for the function $\theta$. Then, $$HD(\pi_{2*}\hat\mu_\psi) \ge \frac{h(\mu_\psi) - \frac 1q \int_{\Sigma_m^+}\ \theta \ d\mu_\psi}{|\chi_s(\hat\mu_\psi)|}.$$
\end{cor}

For self-conformal measures, we obtain from Theorem \ref{bernou} and Corollary \ref{qint} the following:

\begin{cor}\label{scm}
In the setting of Corollary \ref{qint}, let an arbitrary Bernoulli measure $\mu_{\bf p}$ given by the probability vector $\bf p$, and denote its associated self-conformal measure $\nu_{\bf p}$ on the limit set $\Lambda$. Recall the notation in (\ref{mj}). Then the dimension of the self-conformal measure $\nu_{\bf p}$ satisfies:

$$HD(\nu_{\bf p}) \ge \frac{-\mathop{\sum}\limits_{1\le i \le m} p_i\log p_i - \frac{1}{q}\cdot \mathop\sum\limits_{1 \le i \le s}\log m_i \cdot \mathop{\sum}\limits_{[j_1, \ldots, j_q] \in G_i} p_{j_1}\ldots p_{j_q}}{|\chi(\mu_p)|}.$$ 
\end{cor}

\

As an example, consider  the non-linear system given by the contractions
$$F_j(x) = \lambda x+ \vp x^2 +\vp x^3 + \lambda j, \ j \in \{0, 1, 3\},$$ where $\lambda \in [\frac 14, \frac 13]$ and $\vp\ge0$ is sufficiently small.
The limit set $\Lambda_{\lambda, \vp}$ of this system is contained in the interval $[0, \frac{3\lambda}{1-\lambda}+\delta(\vp)]$, for some $\delta(\vp) \to 0$ when $\vp \to 0$. Notice that only $F_0(I_\lambda)$ and $F_1(I_\lambda)$ intersect, so in this case we can take $q=1$ for instance. This can be used in the estimate from Corollary \ref{qint} above, with $q=1$, $k_1 = 2, k_2 = 3$. 
Thus for any probability vector $\textbf{p} = (p_1, p_2, p_3)$,  the self-conformal measure  $\nu_{\bf p}$ on $\Lambda_{\lambda, \vp}$ satisfies:
$$HD(\nu_{\bf p}) \ge \frac{-p_1\log p_1 - p_2\log p_2 - p_3 \log p_3 - \log 2 (p_2+p_1)}{|\log\lambda| + \delta(\vp)},$$ where $\delta(\vp) \mathop{\longrightarrow}\limits_{\vp \to 0} 0$. This estimate can be improved by increasing $q$ for fixed $\lambda$.

\

\ \ \textbf{3.} \ Another application is to order-reversing $\pi_2$-projection measures on \textbf{mixed Julia sets}.  \ 
Mixed Julia sets appear as limit sets of iterated function systems formed with the inverse branches of finitely many rational maps which are expanding on a common open set $V \subset \mathbb C$. Indeed, let us consider $m$ rational maps $R_1, \ldots, R_m$ and assume that their respective Julia sets $J(R_1), \ldots, J(R_m)$ are contained in $V$, and that all $R_j$ are expanding on $V$, for $j = 1, \ldots, m$. 
Now assume the degree of $R_j$ is equal to $d_j$, for $j = 1, \ldots, m$ and that $f_{j, k}, k = 1, \ldots, d_j$ denote the inverse branches $R^{-1}_{j, k}$ of $R_j$ on $V$, for $j = 1, \ldots, m$. Since we assumed that the rational maps $R_j$ are expanding, we obtain a conformal iterated function system consisting of contractions, $$\mathcal S = \{f_{j, k}, 1\le k \le d_j, 1\le j \le m\}$$
We obtain thus a limit set $J(R_1, \ldots, R_m)$ of the system $\mathcal S$, which we call a \textbf{mixed Julia set}. Then if $\psi$ is a H\"older continuous potential on $\Sigma_{d_1+\ldots +d_m}^+$ and $\hat \mu_\psi$ is the equilibrium measure of $\psi\circ \pi_1$ on $\Sigma_m^+\times J(R_1, \ldots, R_m)$, and if $\nu_{2, \psi}$ is the $\pi_2$-projection, $\nu_{2, \psi} = \pi_{2*}\hat\mu_\psi$ on $J(R_1, \ldots, R_m)$, then we obtain from Theorem 1 the exact dimensionality of $\nu_{2, \psi}$ and its dimension:

\begin{cor}\label{mixedJulia}
In the above setting, let $R_j$ be a rational map of degree $d_j$ for any $j = 1, \ldots, m$, and suppose that all the maps $R_j$ are expanding on their respective Julia sets which are  contained in an open set $V\subset \mathbb C$. Consider the system formed by the inverse branches of the maps $R_j$,  $$\mathcal S = \{f_{j, k}, \ k = 1, \ldots, d_j, j = 1, \ldots, m\}$$  Let  $\psi$ H\"older continuous potential on $\Sigma_m^+$, and $\hat\mu_\psi$ be the equilibrium state of $\psi\circ \pi_1$. Then  $\nu_{2, \psi} := \pi_{2*}\hat\mu_\psi$ is exact dimensional on $J(R_1, \ldots, R_m)$, and $HD(\nu_{2, \psi})$ is given by  Theorem \ref{thm1}.
\end{cor}

As an example,  consider the rational maps $$R_j(z) = \gamma_j z^{d_j} + \vp_1 z^{d_j -1}  + \ldots + \vp_{d_j-1} z + c_j, \ j = 1, \ldots, m$$ If $|\gamma_j| = 1$ and $|\vp_1|, \ldots, |\vp_{d_j-1}|, |c_j|$ are all small enough for $1 \le j \le m$, then  the rational maps $R_j$ are expanding on an open neighbourhood $V$ of the unit circle $S^1$, and thus we obtain a contractive IFS consisting of  their inverse branches on $V$, $\mathcal S = \{f_{j, k}, \ k = 1, \ldots, d_j, j = 1, \ldots, m\}$.
Now let us take a H\"older continuous potential $\psi$ on $\Sigma_{d_1+\ldots d_m}^+$ and let $\hat\mu_\psi$ be the equilibrium measure of $\psi\circ \pi_1$. Then the dimension of $\nu_{2, \psi}$ on the mixed Julia set can be computed by Corollary \ref{mixedJulia}. 


\section{Proofs.}

Recall the setting from Section 1, where $\psi$ is a H\"older continuous potential on $\Sigma_I^+$, $\mu^+$ is the equilibrium measure of $\psi$ on $\Sigma_I^+$, and $\hat\mu$ denotes the equilibrium measure $\mu_{\hat\psi}$ of $\hat\psi:=\psi\circ \pi_1$ on $\Sigma_I^+\times \Lambda$. 
Consider the measurable partition $\xi$ of $\Sigma_I^+\times \Lambda$ with the fibers of the projection $\pi_1: \Sigma_I^+\times \Lambda \to \Lambda$, and the associated conditional measures $\mu_\omega$ of $\hat \mu= \mu_{\hat\psi}$ defined for $\mu^+$-a.e $\omega \in \Sigma_I^+$ (see \cite{Ro}); from above, $\mu^+ = \pi_{1*}\hat\mu$. For $\mu^+$-a.e $\omega \in \Sigma_I^+$, the conditional measure $\mu_\omega$ is defined on $\pi_1^{-1}\omega = \{\omega\}\times\Lambda$.
It is clear that the factor space $\Sigma_I^+\times \Lambda/\xi$ is equal to $\Sigma_I^+$, and the  factor measure of $\hat\mu$ satisfies, $$\hat\mu_\xi(A) = \hat\mu(A\times \Lambda) = \mu^+(A),$$
for any measurable set $A \subset \Sigma_I^+$. Thus $\hat\mu_\xi = \mu^+$.
For any borelian set $E$ in $\Sigma_I^+\times \Lambda$ we have, 
\begin{equation}\label{desint}
\hat\mu(E) = \int_{\pi_1E}(\int_{\{\omega\}\times\Lambda}\chi_E d\mu_\omega) \ d\mu^+(\omega) = \int_{\pi_1E} \mu_\omega(E \cap \{\omega\}\times\Lambda) \ d\mu^+(\omega)
\end{equation}
For a Borel set $A$ in $\Lambda$, we have for $\mu^+$-a.e $\omega \in \Sigma_I^+$, 
\begin{equation}\label{muomega}
\mu_\omega(A) = \mathop{\lim}\limits_{n\to \infty}\frac{\hat\mu([\omega_1\ldots \omega_n]\times A)}{\mu^+([\omega_1\ldots \omega_n])}
\end{equation}

\textit{Notation.} 
Two quantities $Q_1, Q_2$ are called \textit{comparable}, denoted by $Q_1 \approx Q_2$, if $\exists C>0$ independent of  parameters in $Q_1, Q_2$, with  $\frac 1C Q_1 \le Q_2 \le C Q_1$. 
$\hfill\square$

The above conditional measures $\mu_\omega$ are defined on $\{\omega\}\times\Lambda$, so they can be considered 
 on $\Lambda$. We now  compare $\mu_\omega(A)$ with $\mu_\eta(A)$.

\begin{lem}\label{prodstr}
There exists a constant $C>0$ so that  for $\mu^+$-a.e $\omega, \eta \in \Sigma_I^+$ and any Borel set $A \subset \Lambda$, 
$\frac 1C \mu_\eta(A) \le \mu_\omega(A) \le C\mu_\eta(A)$.
For any Borel sets $A_1 \subset \Sigma_I^+, A_2 \subset \Lambda$, and $\mu^+$-a.e $\omega\in \Sigma_I^+$, we have:
$$\frac 1C \mu^+(A_1)\cdot\mu_\omega(A_2) \le \hat\mu(A_1 \times A_2) \le C \mu^+(A_1) \cdot \mu_\omega(A_2)$$
In particular there is a constant $C>0$ such that, for $\mu^+$-a.e $\omega \in \Sigma_I^+$ and any Borel set $A\subset \Lambda$,  
$$\frac 1C \mu_\omega(A) \le \nu_2(A) \le C \mu_\omega(A)$$  
\end{lem} 

\begin{proof}
First recall formula (\ref{muomega}) for the conditional measure $\mu_\omega$. From the $\Phi$-invariance of $\hat\mu$, 
\begin{equation}\label{inv}
\hat\mu([\omega_1\ldots \omega_n]\times A) = \mathop{\sum}\limits_{i \in I} \hat\mu([i\omega_1\ldots \omega_n]\times \phi_i^{-1}A)
\end{equation}
Now we can cover the set $A$ with small disjoint balls (modulo $\hat\mu$), so it is enough to consider such a small ball $B = A \subset \Lambda$. The general case will follow then from this.

Recall that for any $i_1, \ldots, i_n \in I, n \ge 1$,  $\phi_{i_1\ldots i_n} := \phi_{i_1}\circ \ldots \circ \phi_{i_n}$. We have Bounded Distortion Property, due to conformal contractions $\phi_i$; i.e $\exists$ a constant $C>0$ so that for any $x, y, n, i_1, \ldots, i_n$, we have $|\phi'_{i_1\ldots i_n}(x)| \le C |\phi_{i_1\ldots i_n}'(y)|$. Since the contractions $\phi_i$ are conformal, let  $i_1, \ldots, i_p \in I$ such that $\phi_{i_p}^{-1}\ldots \phi_{i_1}^{-1}B = \phi_{i_1\ldots i_p}^{-1}B$ is a ball $B(x_0, r_0)$ of a fixed radius $r_0$.
In this way we inflate $B$ along any backward trajectory $\underline i = (i_1, i_2, \ldots) \in \Sigma_I^+$ up to some maximal order $p(\underline i) \ge 1$, so that $\phi_{i_1\ldots i_{p(\underline i)}}^{-1} B $ contains a ball of radius $C_1 r_0$ and it is contained in a ball of radius $ r_0$, for a constant $C_1$ independent of $B, \underline i$. Then by using successively the $\Phi$-invariance of $\hat\mu$, relation (\ref{inv}) becomes:
\begin{equation}\label{inve}
\hat\mu([\omega_1\ldots \omega_n]\times B) = \mathop{\sum}\limits_{\underline i \in I} \hat\mu\big([i_{p(\underline i)}\ldots i_1 \omega_1\ldots \omega_n]\times \phi_{i_1\ldots i_{p(\underline i)}}^{-1}B\big)
\end{equation}
Without loss of generality one can assume that  $\phi_{i_1\ldots i_{p(\underline i)}}^{-1} B $ is a ball of radius $r_0$. Notice that the set $[i_{p(\underline i)}\ldots i_1\omega_1\ldots \omega_n]\times \phi_{i_1\ldots i_{p(\underline i)}}^{-1}B$ is the Bowen ball $[i_{p(\underline i)}\ldots i_1\omega_1\ldots \omega_n] \times B(x_0, r_0)$ for $\Phi$. Since $\hat\mu$ is the equilibrium state of $\psi \circ \pi_1$, and since $P_\Phi(\psi \circ \pi_1) = P_\sigma(\psi):= P(\psi)$, we have:
\begin{equation}\label{Bx0}
\begin{aligned}
\hat\mu([i_{p(\underline i)}\ldots i_1\omega_1\ldots \omega_n] \times \phi_{i_1\ldots i_{p(\underline i})}^{-1}B) &\approx \exp(S_{n+p(\underline i)}\psi(i_{p(\underline i)}\ldots i_1\omega_1\ldots \omega_n)-(n+p(\underline i))P(\psi))\approx\\
&\approx \mu^+([\omega_1\ldots \omega_n]) \cdot \hat \mu([i_{p(\underline i)}\ldots i_1]\times \phi_{i_1\ldots i_{p(\underline i)}}^{-1}B),
\end{aligned}
\end{equation}
where the comparability constant does not depend on $B, i_1, \ldots, i_{p(\underline i)}, n$.
For another $(\eta_1, \ldots, \eta_n) \in I^n$, take again for any $\underline i \in \Sigma_I^+$ the same indices $i_1, \ldots, i_{p(\underline i)}$ s.t $\phi_{i_1\ldots i_{p(\underline i)}}^{-1}B$ is a ball of radius $r_0$, thus,
\begin{equation}\label{Bx1}
\begin{aligned}
\hat\mu([i_{p(\underline i)}\ldots i_1\eta_1\ldots \eta_n]\times \phi_{i_1\ldots i_{p(\underline i)}}^{-1} B) &\approx \exp(S_{n+p(\underline i)}\psi(i_{p(\underline i)}\ldots i_1\eta_1\ldots \eta_n) - (n+p(\underline i)) P(\psi)) \\
&\approx \mu^+([\eta_1\ldots \eta_n]) \cdot \hat\mu([i_{p(\underline i)}\ldots i_1]\times \phi_{i_1\ldots i_{p(\underline i)}}^{-1}B),
\end{aligned}
\end{equation}
where the comparability constant does not depend on $B, i_1, \ldots, i_{p(\underline i)}, n$.
But the cover of $A$ with small balls  of type $B$ and the above process of inflating these balls along prehistories $\underline i$ to balls of radius $r_0$, can be done along any trajectories $\omega, \eta$. 
Thus by  (\ref{inve}) and using the uniform estimates (\ref{Bx0}), (\ref{Bx1}) and (\ref{muomega}), we obtain that there exists a constant $C>0$ such that for $\mu^+$-a.e $\omega, \eta \in \Sigma_I^+$, 
\begin{equation}\label{etao}
\frac 1C\mu_\eta(A) \le \mu_\omega(A) \le C\mu_\eta(A)
\end{equation}
From (\ref{etao}) and (\ref{desint}) for $\hat\mu$, $\exists$ a constant $C$ so that for any Borel sets $A_1\subset \Sigma_I^+, A_2\subset \Lambda$, 
 $$
\frac 1C \mu^+(A_1)\cdot\mu_\omega(A_2)\le \hat\mu(A_1\times A_2) \le C \mu^+(A_1) \cdot \mu_\omega(A_2)
$$
To finish the proof, recall that $\nu_2 = \pi_{2*}\hat\mu$, so $\nu_2(A) = \hat\mu(\Sigma_I^+\times A)$, and use the last estimates.

\end{proof}

\textit{Proof of Theorem \ref{thm1}.}

First, we prove the upper estimate for the pointwise dimension of $\nu_2$. For any $n \ge 1, (\omega, x) \in \Sigma_I^+\times \Lambda$, \ $\Phi^n(\omega, x) = (\sigma^n\omega, \phi_{\omega_n\ldots \omega_1}(x))$. From Birkhoff Ergodic Theorem applied to the $\Phi$-invariant measure $\hat\mu$, it follows that for $\hat\mu$-a.e $(\omega, x) \in \Sigma_I^+\times \Lambda$, 
$$\frac 1n \log|\phi_{\omega_n\ldots \omega_1}'|(x) \mathop{\longrightarrow}\limits_n \mathop{\sum}\limits_{i \in I} \int_{[i]\times \Lambda} \log|\phi_i'(x)| \ d\hat\mu(\omega, x) = \chi_s(\hat\mu)$$
On the other hand, from the Chain Rule for Jacobians, Birkhoff Ergodic Theorem and the formula for folding entropy $F_\Phi(\hat\mu)$, it follows that for $\hat\mu$-a.e $(\omega, x) \in \Sigma_I^+\times \Lambda$, 
$$\frac 1n \log J_{\Phi^n}(\hat\mu)(\omega, x) \mathop{\longrightarrow}\limits_n F_\Phi(\hat\mu)$$
 Thus for a set of $(\omega, x) \in\Sigma_I^+\times \Lambda$ of  full $\hat\mu$-measure,  
$$\frac 1n \log|\phi_{\omega_n\ldots \omega_1}'|(x) \mathop{\longrightarrow}\limits_n \mathop{\sum}\limits_{i \in I} \int_{[i]\times \Lambda} \log|\phi_i'(x)| \ d\hat\mu(\omega, x) = \chi_s(\hat\mu), \ \text{and} \  \ \frac 1n \log J_{\Phi^n}(\hat\mu)(\omega, x) \mathop{\longrightarrow}\limits_n F_\Phi(\hat\mu)$$
We now want to prove that the Jacobian $J_{\Phi^n}(\hat\mu)(\omega, x)$ depends basically  only on $\omega_1, \ldots, \omega_n$, i.e there exists a constant $C>0$ such that for every  $n \ge 1$, and $\hat\mu$-a.e  $(\eta, x)  \in [\omega_1\ldots \omega_n] \times \Lambda$, 
\begin{equation}\label{compjac}
\frac 1C J_{\Phi^n}(\hat\mu)(\eta, x) \le J_{\Phi^n}(\hat\mu)(\omega, x) \le C J_{\Phi^n}(\hat\mu)(\eta, x)
\end{equation}
In order to prove this, notice that if $r>0$, and $p>1$ is such that $diam[\omega_1\ldots \omega_{n+p}] = r$, then 
\begin{equation}\label{jacphin}
J_{\Phi^n}(\hat \mu)(\omega, x) = \mathop{\lim}\limits_{r\to 0, p \to \infty} \frac{\hat\mu(\Phi^n([\omega_1\ldots\omega_{n+p}]\times B(x, r))}{\hat\mu([\omega_1\ldots \omega_{n+p}]\times B(x, r))}
\end{equation}
 In our case, $\Phi^n([\omega_1\ldots \omega_{n+p}]\times B(x, r)) = [\omega_{n+1}\ldots \omega_{n+p}]\times \phi_{\omega_n\ldots\omega_1}B(x, r)$. If $\eta \in [\omega_1\ldots \omega_n]$,  
 $$\Phi^n([\eta_1\ldots \eta_{n+p}] \times B(x, r)) = [\eta_{n+1}\ldots \eta_{n+p}]\times \phi_{\omega_n \ldots \omega_1} B(x, r)$$
But from Lemma \ref{prodstr} there exists a constant $C>0$ such that for $\mu^+$-a.e $\omega \in \Sigma_I^+$, and any $n, p\ge 1$, 
\begin{equation}\label{comphat}
\begin{aligned}
\frac 1C \mu^+([\omega_{n+1}\ldots \omega_{n+p}])&\mu_\omega(\phi_{\omega_n\ldots\omega_1}B(x, r)) \le \hat \mu([\omega_{n+1}\ldots \omega_{n+p}]\times \phi_{\omega_n\ldots \omega_1}B(x, r)) \le \\
&\le C \mu^+([\omega_{n+1}\ldots \omega_{n+p}])\mu_\omega(\phi_{\omega_n\ldots \omega_1} B(x, r)),
\end{aligned}
\end{equation}
and similarly for $\hat\mu([\eta_{n+1}\ldots \eta_{n+p}]\times \phi_{\omega_n\ldots \omega_1} B(x, r))$.
Hence in view of (\ref{jacphin}) and (\ref{comphat}), we have only to compare the following quantities,
$$\frac{\mu^+([\omega_{n+1}\ldots\omega_{n+p}])\cdot \mu_\omega(\phi_{\omega_n\ldots \omega_1} B(x, r))}{\mu^+([\omega_1\ldots \omega_{n+p}])\cdot \mu_\omega(B(x, r))} \ \ \text{and} \ \ \frac{\mu^+([\eta_{n+1}\ldots \eta_{n+p}])\cdot\mu_\omega(\phi_{\omega_n\ldots\omega_1} B(x, r))}{\mu^+([\eta_1\ldots \eta_{n+p}]) \cdot \mu_\omega(B(x, r))}$$
However recall that $\eta \in [\omega_1\ldots \omega_n]$, thus there exists a constant $K>0$ such that 
\begin{equation}\label{holderS}
|S_n\psi(\eta_1\ldots \eta_n\ldots) - S_n\psi(\omega_1\ldots \omega_n\ldots)| \le K,
\end{equation}
since $\psi$ is H\"older continuous and $\sigma$ is expanding on $\Sigma_I^+$.
The same argument also implies that  \ $S_{n+p}\psi(\omega_1\ldots \omega_{n+p}\ldots)$ is determined in fact only by the first $n+p$ coordinates (modulo an additive constant).
Since $\mu^+$ is the equilibrium measure of $\psi$ on $\Sigma_I^+$, and thus  a Gibbs measure, we obtain: 
\begin{equation}\label{mu+o}
\begin{aligned}
&\frac{\mu^+([\omega_{n+1}\ldots \omega_{n+p}])}{\mu^+([\omega_1\ldots \omega_{n+p}])} \approx \frac{\exp(S_p\psi(\omega_{n+1}\ldots \omega_{n+p}\ldots) - pP(\psi))}{\exp(S_{n+p}\psi(\omega_1\ldots \omega_{n+p}) - (n+p)P(\psi))} \ \text{and},\\
&\frac{\mu^+([\eta_{n+1}\ldots \eta_{n+p}])}{\mu^+([\eta_1\ldots \eta_{n+p}])} \approx \frac{\exp(S_p\psi(\eta_{n+1}\ldots \eta_{n+p}\ldots) - pP(\psi))}{\exp(S_{n+p}\psi(\eta_1\ldots \eta_{n+p}) - (n+p)P(\psi))},
\end{aligned}
\end{equation}
 where the comparability constant does not depend on $n, p, \omega, \eta$. But we have: \ $S_{n+p}\psi(\omega_1\ldots \omega_{n+p}\ldots) = S_n\psi(\omega_1\ldots \omega_{n+p}\ldots) + S_p\psi(\omega_{n+1}\ldots \omega_{n+p}\ldots)$.
 And similary for $S_{n+p}\psi(\eta_1\ldots \eta_{n+p}\ldots)$.
Therefore, using (\ref{jacphin}), (\ref{comphat}), (\ref{holderS}) and (\ref{mu+o}), we obtain the Jacobians inequalites in (\ref{compjac}).

Let us take now, for any $n >1$ and $\vp>0$,  the Borel set in $\Sigma_I^+\times\Lambda$: 
$$
\begin{aligned}
A(n, \vp):= \{&(\omega, x)\in \Sigma_I^+ \times \Lambda, \ \big|\frac{\log|\phi_{\omega_n\ldots \omega_1}'(x)|}{n}-\chi_s(\hat \mu)\big| < \vp, \ |\frac{\log J_{\Phi^n}(\hat\mu)(\omega, x)}{n} - F_\Phi(\hat\mu)|<\vp,  \\ 
&\text{and} \ |\frac{S_n\psi(\omega)}{n}-\int\psi d\mu^+|<\vp\}
\end{aligned}
$$
Then from Birkhoff Ergodic Theorem, for any $\vp>0$, $\hat\mu(A(n, \vp)) \mathop{\longrightarrow}\limits_{n\to \infty} 1$. From (\ref{compjac}), if $(\omega, x) \in A(n, \vp)$ and $\eta \in [\omega_1\ldots \omega_n]$, then $(\eta, x) \in A(n, 2\vp)$, so for any $\delta>0$, 
$$
\mu^+(\{\omega \in \Sigma_I^+, \nu_2(\pi_2(A(n, \vp)\cap [\omega_1\ldots \omega_n]\times \Lambda))>1-\delta\}) \mathop{\longrightarrow}\limits_{n} 1$$
We have from (\ref{compjac}) that $[\omega_1\ldots\omega_n] \times \pi_2 A(n, \vp)\subset A(n, 2\vp)$. Thus from Lemma \ref{prodstr}, for $n > n(\delta)$, 
 \begin{equation}\label{hatmuA}
 \hat\mu(A(n, 2\vp)\cap [\omega_1\ldots \omega_n]\times \Lambda) > C(1-\delta)\cdot\hat\mu([\omega_1\ldots \omega_n]\times \Lambda)
 \end{equation}

Let now $r_n:= 2\vp|\phi_{\omega_n\ldots \omega_1}'(x)|$, \ for $(\omega, x) \in A(n, \vp) \cap [\omega_1\ldots \omega_n]\times \Lambda$. With $y = \phi_{\omega_n\ldots \omega_1}(x)$ we have $\nu_2(B(y, r_n)) \ge \nu_2(\phi_{\omega_n\ldots \omega_1}(\pi_2 (A(n, 2\vp) \cap [\omega_1\ldots \omega_n]\times \Lambda))$, and then since $\nu_2 = \pi_{2*}\hat \mu$, we obtain
$$\nu_2(B(y, r_n)) \ge \hat\mu(\Sigma_I^+\times \pi_2(\Phi^n(A(n, 2\vp)\cap [\omega_1\ldots \omega_n]\times \Lambda))) \ge \hat\mu(\Phi^n(A(n, \vp)\cap [\omega_1\ldots \omega_n]\times \Lambda))$$
But  $\Phi^n$ is injective on the cylinder $[\omega_1\ldots \omega_n]\times \Lambda$ since $\phi_j$ are injective. Thus we can apply the Jacobian formula for the measure of the $\Phi^n$-iterate in the last term of last inequality above,
\begin{equation}\label{injf}
\begin{aligned}
&\nu_2(B(y, r_n)) \ge \hat\mu(\Phi^n(A(n, \vp)\cap [\omega_1\ldots \omega_n]\times \Lambda)) = \int_{A(n, \vp)\cap [\omega_1\ldots \omega_n]\times \Lambda} J_{\Phi^n}(\hat\mu) \ d\hat\mu \\
&\ge \exp\big(n(F_\Phi(\hat\mu)- \vp)\cdot \hat\mu(A(n, \vp) \cap [\omega_1\ldots \omega_n]\times \Lambda\big)
\end{aligned}
\end{equation} 
Recall that $\mu^+([\omega_1\ldots \omega_n]) = \hat \mu([\omega_1\ldots \omega_n]\times \Lambda)$. Then, from (\ref{injf}) and (\ref{hatmuA}) we obtain:
\begin{equation}\label{lastpager}
\nu_2(B(y, r_n))\ge \exp(n(F_\Phi(\hat\mu)-\vp))(1-\delta)C \mu^+([\omega_1\ldots \omega_n]) \ge C(1-\delta) e^{n(F_\Phi(\hat\mu)-\vp)}\exp(S_n\psi(\omega) - nP(\psi)),
\end{equation}
where $C$ is independent of $n, \omega, x, y$, and $P(\psi):=P_\sigma(\psi)$. Since $\mu^+$ is the equilibrium state of $\psi$,  $$P(\psi) = h(\mu^+) + \int \psi \ d\mu^+$$
But from the definition of $A(n, \vp)$, for any $(\omega, x) \in A(n, \vp)$, we have
\begin{equation}\label{nchi}
e^{n(\chi_s(\hat\mu) +\vp)} \ge  r_n=2\vp|\phi_{\omega_n\ldots\omega_1}(x)| \ge e^{n(\chi_s(\hat\mu) - \vp)}, \  \text{thus}, \  
n (\chi_s(\hat\mu)+\vp) \ge \log r_n \ge n(\chi_s(\hat\mu) - \vp)
\end{equation}
From (\ref{lastpager}), (\ref{nchi}) and the above formula for pressure, we obtain that for $\nu_2$-a.e $y \in \Lambda$,
\begin{equation}\label{upperestd}
\overline{\delta}(\nu_2)(y) = \mathop{\overline{\lim}}\limits_{r\to 0}\frac{\log \nu_2(B(y, r))}{\log r} \le \frac{F_\Phi(\hat\mu)-h(\mu^+)}{\chi_s(\hat\mu)} = \frac{h(\mu^+)-F_\Phi(\hat\mu)}{|\chi_s(\hat\mu)|},
\end{equation}
which proves the upper estimate for the (upper) pointwise dimension of $\nu_2$.

\

Now, we prove the more difficult \textit{lower estimate} for the pointwise dimension of $\nu_2$.
\newline
Define for any $m \ge 1$ and $\vp>0$, the following Borel set in $\Sigma_I^+\times \Lambda$, 
$$
\begin{aligned}
\tilde A(m, \vp) := \big\{(\omega, x) \in \Sigma_I^+\times \Lambda, &\ |\frac 1n\log|\phi_{\omega_n\ldots \omega_1}'(x)|-\chi_s(\hat\mu)|<\vp, \text{and} \ |\frac 1n \log J_{\Phi^n}(\hat\mu)(\omega, x)- F_{\Phi^n}(\hat\mu)| < \vp, \\
& \text{and} \ |\frac 1n S_n\psi(\omega) - \int \psi d\mu^+| < \vp, \ \forall n \ge m\big\}
\end{aligned}
$$ 
We know from Birkhoff Ergodic Theorem that $\hat\mu(\tilde A(m, \vp)) \mathop{\to}\limits_m 1$, for any $\vp>0$. So we obtain $\nu_2(\pi_2\tilde A(m, \vp)) \mathop{\to}\limits_m 1$. But $\Phi^n([i_1\ldots i_n]\times \Lambda) = \Sigma_I^+\times \phi_{i_n\ldots i_1}\Lambda$, and from the $\Phi$-invariance of $\hat\mu$, we have $\hat\mu(\Phi^n([i_1\ldots i_n]\times \Lambda) \ge \hat\mu([i_1\ldots i_n] \times \Lambda)$. Moreover $\hat\mu([i_1\ldots i_n] \times \Lambda) >0$,  since $\hat\mu$ is the equilibrium measure of $\psi\circ \pi_1$ and $[i_1\ldots i_n]\times \Lambda$ is an open set. In conclusion,  
\begin{equation}\label{pozm}
\nu_2(\phi_{i_n\ldots i_1}\Lambda) = \hat\mu(\Sigma_I^+\times \phi_{i_n\ldots i_1}\Lambda) = \hat\mu(\Phi^n([i_1\ldots i_n]\times \Lambda) \ge \hat\mu([i_1\ldots i_n] \times \Lambda) > 0
\end{equation}

We present briefly the general strategy of the proof, which will be detailed in the sequel. Since $\Phi^m([i_1\ldots i_m]\times \Lambda) = \Sigma_I^+\times \phi_{i_m\dots i_1}\Lambda$, and $\nu_2 = \pi_{2*}\hat\mu$, we have: $$\nu_2(\phi_{i_m\ldots i_1}\Lambda) = \hat\mu(\Sigma_I^+\times \phi_{i_m\ldots i_1}\Lambda) = \hat\mu(\Phi^m([i_1\ldots i_m]\times \Lambda)$$ 
Notice that a small ball $B(x, r)$ can intersect many sets of type $\phi_{i_m\ldots i_1}\Lambda$, for various $m$-tuples $(i_1, \ldots, i_m) \in I^m$, and these image sets may also intersect one another. Thus when estimating $\nu_2(B(x, r))$, all of these sets must be considered; it is not enough in principle to consider only one intersection $B(x, r) \cap \phi_{i_m\ldots i_1}\Lambda$.  \ However, we know from (\ref{pozm}) that $\nu_2(\phi_{i_m\ldots i_1}\Lambda) >0$, thus from the Borel Density Theorem (see \cite{Pe}), it follows that for $\nu_2$-a.e $x \in \phi_{i_m\ldots i_1}\Lambda$, and for all $0< r < r(x)$, $$\frac{\nu_2(B(x, r)\cap \phi_{i_m\ldots i_1}\Lambda)}{\nu_2(B(x, r))} > 1/2$$
 Hence, from the point of view of the measure $\nu_2$, the intersection $B(x, r) \cap \phi_{i_m\ldots i_1}\Lambda$ contains at least half of the $\nu_2$-measure of the ball $B(x, r)$. This hints that it is enough to consider only one good image set of type $\phi_{i_m\ldots i_1}\Lambda$. Then  since $\nu_2(\pi_2\tilde A(m, \vp)) \mathop{\to}\limits_m 1$, we can consider only $\nu_2(\phi_{i_m\ldots i_1}(\pi_2\tilde A(m, \vp)))$, which can be estimated using the Jacobian $J_{\Phi^m}(\hat\mu)$ and the genericity of points in $\tilde A(m, \vp)$ with respect to  $\log J_{\Phi^m}(\hat\mu)$ and $\log|\phi_{i_m\ldots i_1}'|$. Then we  repeat this argument whenever the iterate of a point belongs to the set of generic points $\tilde A(m, \vp)$. However not all iterates of a point belong to this set, but it will be shown by a delicate estimate that ''most'' of  them hit $\tilde A(m, \vp)$. For the iterates not in $\tilde A(m, \vp)$, we use a different type of estimate. Then we repeat and combine these two  estimates, by an interlacing procedure. \ 
We now proceed with the full proof:

For any integer $m>1$, consider the Borel set $\tilde A(m, \vp)$ defined above. Then for any $\alpha>0$ arbitrarily small, there exists an integer $m(\alpha)>1$ such that for any $m > m(\alpha)$,  we have:
\begin{equation}\label{Amalpha}
\hat\mu(\tilde A(m, \vp)) > 1-\alpha
\end{equation}
Let us fix such an integer $m > m(\alpha)$. Then from Birkhoff Ergodic Theorem applied to $\Phi^m$ and $\chi_{\tilde A(m, \vp)}$, we have that for $\hat\mu$-a.e $(\omega', x')\in \Sigma_I^+\times \Lambda$, 
$$\frac 1n Card\{0 \le k \le n, \ \Phi^{km}(\omega', x') \in \tilde A(m, \vp)\} \mathop{\longrightarrow}\limits_{n \to \infty} \hat\mu(\tilde A(m, \vp))$$
Hence, there exists an integer $n(\alpha)$ and a Borel set $D(\alpha)\subset \Sigma_I^+\times \Lambda$, with $\hat\mu(D(\alpha)) > 1-\alpha$, such that for $(\omega', x') \in D(\alpha)$ and $n \ge n(\alpha)$, we have:
\begin{equation}\label{Dalpha}
\frac 1n Card\{0 \le k \le n, \ \Phi^{mk}(\omega', x') \in \tilde A(m, \vp)\} > 1-2\alpha
\end{equation}
In other words a large proportion of the iterates of points $(\omega', x')$ in $D(\alpha)$, belong to the set of generic points $\tilde A(m, \vp)$.
So in the $\Phi^m$-trajectory $(\omega', x'), \Phi^m(\omega', x'), \ldots, \Phi^{nm}(\omega', x')$, there are at least $(1-2\alpha)n$ iterates in $\tilde A(m, \vp)$.  \

 For arbitrary indices $i_1, \ldots  i_m \in I$, let us define now the Borel set in $\Lambda$, $$Y(i_1, \ldots, i_m):= \phi_{i_1 \ldots i_m}\pi_2(\tilde A(m, \vp))$$
Consider first the intersection of all these sets, namely $\mathop{\bigcap}\limits_{i_1, \ldots, i_m \in I} Y(i_1, \ldots, i_m)$. Then take the intersections of these sets except only one of them, so consider the sets of type 

$\mathop{\bigcap}\limits_{(j_1, \ldots, j_m) \in I^m \setminus  \{(i_1, \ldots, i_m)\}} Y(j_1, \ldots, j_m) \setminus Y(i_1, \ldots, i_m),$ for all $(i_1, \ldots, i_m) \in I^m$. Then consider the intersections of all the sets $Y(j_1, \ldots, j_m)$ excepting two of them, namely the intersections of type $\mathop{\bigcap}\limits_{(j_1, \ldots, j_m) \in I^m \setminus \{(i_1, \ldots, i_m),  (i_1', \ldots, i_m')\}} Y(j_1, \ldots, j_m) \setminus \big(Y(i_1, \ldots, i_m) \cup Y(i_1', \ldots, i_m')\big)$, for all the $m$-tuples $(i_1, \ldots, i_m),$ $(i_1', \ldots, i_m')\in I^m$. We continue this procedure until we exhaust all the possible intersections of type 
$$\mathop{\bigcap}\limits_{(j_1, \ldots, j_m) \in I^m \setminus \mathcal J} Y(j_1, \ldots, j_m) \setminus \mathop{\bigcup}\limits_{(i_1, \ldots, i_m) \in \mathcal J} Y(i_1, \ldots, i_m),$$ for some arbitrary given set $\mathcal J$ of $m$-tuples  from $I^m$.  Notice that in this way, by taking all the subsets $\mathcal J \subset I^m$, we obtain by the above procedure mutually disjoint Borel sets (some may be empty).  Denote these  mutually disjoint nonempty sets obtained above by $Z_1(\alpha; m, \vp), \ldots, Z_{M(m)}(\alpha; m, \vp)$. 

Now if for some $1\le i \le M(m)$, we know that $\nu_2(Z_i(\alpha; m, \vp)) >0$, then from the Borel Density Theorem (see \cite{Pe}),  there exists a Borel subset $G_i(\alpha; m, \vp) \subset Z_i(\alpha; m, \vp)$, with $$\nu_2(G_i(\alpha; m, \vp)) \ge \nu_2(Z_i(\alpha; m, \vp))(1 - \alpha),$$ and  there exists $r_i(\alpha; m, \vp)>0$, such that for any $x \in G_i(\alpha; m, \vp)$ and for any $0 <r <r_i(\alpha; m, \vp)$, 
\begin{equation}\label{Gi}
\frac{\nu_2(B(x, r) \cap Z_i(\alpha; m, \vp))}{\nu_2(B(x, r))} > 1-\alpha.
\end{equation}

Now define the Borel subset of $\Lambda$, 
$$G(\alpha; m, \vp):= \mathop{\bigcup}\limits_{i=1}^{M(m)} G_i(\alpha; m, \vp)$$
From the construction of the mutually disjoint sets $Z_i(\alpha; m, \vp)$, it follows that,
\begin{equation}\label{cupZ}
\mathop{\sum}\limits_{1 \le i \le M(m)} \nu_2(Z_i(\alpha; m, \vp)) = \nu_2(\mathop{\cup}\limits_{i_1, \ldots, i_m \in I} Y(i_1, \ldots, i_m))
\end{equation}
But from definition of $Y(i_1, \ldots, i_m)$ and the disjointness of different $m$-cylinders,  we have that,
$$
\begin{aligned}
\mathop{\cup}\limits_{i_1, \ldots, i_m \in I} Y(i_1, \ldots, i_m) & = \mathop{\cup}\limits_{i_1, \ldots, i_m \in I} \pi_2(\Phi^m([i_m\ldots i_1] \times \pi_2(\tilde A(m, \vp)))) =\\
&=  \pi_2(\mathop{\cup}\limits_{i_1, \ldots, i_m \in I} \Phi^m([i_m\ldots i_1]\times \pi_2\tilde A(m, \vp))) = \pi_2(\Phi^m(\Sigma_I^+ \times \pi_2\tilde A(m, \vp)))
\end{aligned}
$$
However since $\nu_2 = \pi_{2*}\hat\mu$, and using the $\Phi$-invariance of $\hat \mu$ and (\ref{Amalpha}), it follows that: $$
\begin{aligned}
\nu_2\big( \pi_2(\Phi^m(\Sigma_I^+ \times \pi_2\tilde A(m, \vp)))\big) &= \hat \mu(\Sigma_I^+ \times  \pi_2(\Phi^m(\Sigma_I^+ \times \pi_2\tilde A(m, \vp)))) \ge \hat\mu(\Phi^m(\Sigma_I^+ \times \pi_2\tilde A(m, \vp))) \ge \\ &\ge \hat\mu(\Sigma_I^+ \times \pi_2\tilde A(m, \vp))  \ge \hat \mu(\tilde A(m, \vp)) > 1-\alpha
\end{aligned}
$$
Thus from the last two displayed formulas, it follows that:
\begin{equation}\label{tildemare}
\nu_2(\mathop{\bigcup}\limits_{i_1, \ldots, i_m \in I} Y(i_1, \ldots, i_m)) > 1-\alpha
\end{equation}
Hence from (\ref{cupZ}) and (\ref{tildemare}), we obtain:
\begin{equation}\label{MmZ}
\nu_2\big(\mathop{\bigcup}\limits_{1 \le i \le M(m)} Z_i(\alpha; m, \vp) \big) > 1-\alpha
\end{equation}
But the Borel sets $Z_i(\alpha; m, \vp), 1 \le i \le M(m)$ are mutually disjoint, and from (\ref{Gi}) we know that $G_i(\alpha; m, \vp) \subset Z_i(\alpha; m, \vp)$ and that, for $1 \le i \le M(m)$, $$\nu_2(G_i(\alpha; m, \vp)) \ge  \nu_2(Z_i(\alpha; m, \vp)) (1-\alpha)$$  Hence from the definition of $G(\alpha; m, \vp) = \mathop{\bigcup}\limits_{1\le i \le M(m)} G_i(\alpha; m, \vp)$ and from (\ref{MmZ}),
\begin{equation}\label{nu2Ga}
 \nu_2(G(\alpha; m, \vp)) > (1-\alpha)^2> 1-2\alpha
 \end{equation}
 Denote now the following intersection set by,
$$X(\alpha; m, \vp):= G(\alpha; m, \vp) \cap \pi_2\tilde A(m, \vp)$$
Then from the above estimate for $\nu_2(G(\alpha; m, \vp))$ and by using (\ref{Amalpha}), we obtain:
\begin{equation}\label{nu2X}
\nu_2(X(\alpha; m, \vp)) > 1-3\alpha
\end{equation}

Now by applying the same argument as in (\ref{Dalpha}) to the set $\tilde A(m, \vp) \cap \big(\Sigma_I^+\times X(\alpha; m, \vp)\big)$, we obtain that there exists a Borel set $\tilde D(\alpha; m, \vp) \subset \tilde A(m, \vp) \subset \Sigma_I^+\times \Lambda$, with 
\begin{equation}\label{Dam}
\hat\mu(\tilde D(\alpha; m, \vp)) > 1- 4\alpha,
\end{equation}
 and such that for any pair $(\ul i, x') \in \tilde D(\alpha; m, \vp)$, at least a number of  $(1-3\alpha)n$ of the points $\pi_2(\ul i, x')$, $\pi_2\Phi^m(\ul i, x'), \ldots, \pi_2\Phi^{nm}(\ul i, x')$ belong to  $X(\alpha; m, \vp)$. Moreover any point $\zeta \in X(\alpha; m, \vp)$ satisfies condition (\ref{Gi}) for all  $0 < r < r_i(\alpha; m, \vp)$ if $\zeta \in Z_i(\alpha; m, \vp)$, for some $1 \le i \le M(m)$. So denote $$r_m(\alpha, \vp) := \mathop{\min}\limits_{1\le i \le M(m)} r_i(\alpha; m, \vp)$$

Consider now $(\ul i, x') \in \tilde D_m(\alpha, \vp)$, and denote by $x = \phi_{i_{nm}\ldots i_1}(x') = \pi_2\Phi^{nm}(\ul i, x')$. So we have the following backward trajectory of $x$ with respect to $\Phi^m$ determined by the sequence $\ul i$ from above,  
\begin{equation}\label{seqx}
x, \phi_{i_{(n-1)m}\ldots i_1}(x'), \ldots, \phi_{i_m\ldots i_1}(x'), x',
\end{equation}
 and denote these points respectively by $x, x_{-m}, \ldots, x_{-(n-1)m}, x_{-nm} = x'$.  
To see the next argument,  assume for simplicity that the first preimage of $x$ in this trajectory, namely $x_{-m} = \phi_{i_{(n-1)m}\ldots i_1}(x')$ belongs to $X(\alpha; m, \vp)$. Then from (\ref{Gi}), Lemma \ref{prodstr}, and the genericity of $J_{\Phi^m}$ on $\tilde A(m, \vp)$,
$$
\begin{aligned}
\nu_2(B(x, r)) &\le \frac{1}{1-\alpha} \nu_2(B(x, r) \cap \phi_{i_{nm}\ldots i_{(n-1)m}}\pi_2\tilde A(m, \vp))  = \\
& = \frac{1}{1-\alpha} \hat\mu(\Sigma_I^+\times (B(x, r) \cap \phi_{i_{nm}\ldots i_{(n-1)m}}\pi_2\tilde A(m, \vp))) \le  \\
& \le \frac{1}{1-\alpha} \hat\mu(\Phi^m([i_{(n-1)m}\ldots i_{nm}]\times (\pi_2\tilde A(m, \vp)\cap  (\phi_{i_{nm}\ldots i_{(n-1)m}})^{-1}B(x, r) ))) = \\ 
&= \frac{1}{1-\alpha}\int_{[i_{(n-1)m}\ldots i_{nm}]\times (\pi_2\tilde A(m, \vp)\cap   (\phi_{i_{nm}\ldots i_{(n-1)m}})^{-1}B(x, r)  )} J_{\Phi^m}(\hat\mu) \ d\hat\mu \le \\
&\le C\frac{1}{1-\alpha} e^{m(F_\Phi(\hat\mu)+\vp)}\mu^+([i_{(n-1)m \ldots i_{nm}}])\cdot \nu_2( (\phi_{i_{nm}\ldots i_{(n-1)m}})^{-1}B(x, r) ) \le \\
&\le \frac{C}{1-\alpha} e^{m(F_\Phi(\hat\mu)+2\vp-h(\mu^+))}\nu_2( (\phi_{i_{nm}\ldots i_{(n-1)m}})^{-1}B(x, r)),
\end{aligned}
$$
where the last inequality follows since $\mu^+$ is the equilibrium state of $\psi$ and $P(\psi) = h_{\mu^+} + \int \psi d\mu^+$.
Thus we obtain from above the following estimate on the measure of $B(x, r)$,
\begin{equation}\label{nu2B}
\nu_2(B(x, r)) \le C \frac{1}{1-\alpha} \cdot e^{m(F_\Phi(\hat\mu)-h(\mu^+)+2\vp)}\cdot \nu_2( (\phi_{i_{nm}\ldots i_{(n-1)m}})^{-1}B(x, r) )
\end{equation}
This argument can be repeated until we reach in the above backward trajectory of $x$ (\ref{seqx}), a preimage  which is \textit{not} in $X(\alpha; m, \vp)$.  Denote then by $k_1\ge 1$ the first integer $k$ for which $x_{-mk} \notin X(\alpha; m, \vp)$, and assume that the above process is interrupted for $k_1'$ indices, namely $\{x_{-mk_1}, x_{-m(k_1+1)}, \ldots, x_{-m(k_1+k_1'-1)}\} \cap X(\alpha; m, \vp) = \emptyset$.  Then denote $y:= x_{-mk_1}$, and let us estimate $\nu_2( (\phi_{i_{nm}\ldots i_{(n-k_1)m}})^{-1}B(x, r) )$.
By definition of the push-forward measure $\nu_2$, 
$$\nu_2( (\phi_{i_{nm}\ldots i_{(n-k_1)m}})^{-1}B(x, r) ) = \hat\mu(\Sigma_I^+\times  (\phi_{i_{nm}\ldots i_{(n-k_1)m}})^{-1}B(x, r) )$$
 By definition of $k_1, k_1'$, we have that $x_{-m(k_1+k_1')} \in X(\alpha; m, \vp)$.
By repeating the estimate in (\ref{nu2B}), we obtain an upper estimate for $\nu_2(B(x, r))$, 
\begin{equation}\label{repeatnu2}
 \nu_2(B(x, r)) \le ( \frac{C}{1-\alpha})^{k_1} \cdot e^{mk_1(F_\Phi(\hat\mu)-h(\mu^+)+2\vp)}\cdot \nu_2( (\phi_{i_{nm}\ldots i_{(n-k_1)m}})^{-1}B(x, r) )
 \end{equation}
Now on the other hand from the $\Phi$-invariance of $\hat\mu$ and  the definition of $\nu_2$,
\begin{equation}\label{Byr1}
\begin{aligned}
\nu_2( (\phi_{i_{nm}\ldots i_{(n-k_1)m}})^{-1} B(x, r) ) &= \hat\mu(\Sigma_I^+\times  (\phi_{i_{nm}\ldots i_{(n-k_1)m}})^{-1}B(x, r) ) = \\
&= \hat\mu(\phi^{-mk_1'}(\Sigma_I^+\times  (\phi_{i_{nm}\ldots i_{(n-k_1)m}})^{-1}B(x, r) ))\\
&= \mathop{\sum}\limits_{j_p\in I, 1\le p \le mk_1'} \hat\mu([j_1\ldots j_{mk_1'}]\times \phi_{j_{mk_1'}\ldots j_1}^{-1}  (\phi_{i_{nm}\ldots i_{(n-k_1)m}})^{-1}B(x, r)  )
\end{aligned}
\end{equation}
Recall that the set of non-generic points satisfies, $$\hat\mu((\Sigma_I^+\times \Lambda ) \setminus \tilde A(m, \vp)) < \alpha$$

We now compare the $\hat\mu$-measure of the set of generic points with respect to the $\hat\mu$-measure, with the $\hat\mu$-measure of the set of non-generic points. 
There are 2 cases: 

\ a)  \ If, $$\hat\mu\big(\tilde A(m, \vp) \bigcap \Phi^{-mk_1'}(\Sigma_I^+\times  (\phi_{i_{nm}\ldots i_{(n-k_1)m}})^{-1}B(x, r))\big) < \frac 12 \hat\mu(\Phi^{-mk_1'}(\Sigma_I^+\times  (\phi_{i_{nm}\ldots i_{(n-k_1)m}})^{-1}B(x, r) )),$$ then non-generic points have more mass than the generic points, hence,
$$
\nu_2( (\phi_{i_{nm}\ldots i_{(n-k_1)m}})^{-1}B(x, r) ) = \hat\mu(\Sigma_I^+ \times  (\phi_{i_{nm}\ldots i_{(n-k_1)m}})^{-1}B(x, r) ) < 2\alpha <<1
$$
By collecting all sets with the above property and taking $\alpha\to 0$, this case is then straightforward.

\ \  b) \  If, \ $$\hat\mu\big(\tilde A(m, \vp) \cap \Phi^{-mk_1'}(\Sigma_I^+\times  (\phi_{i_{nm}\ldots i_{(n-k_1)m}})^{-1}B(x, r) )\big) \ge \frac 12 \hat \mu(\Phi^{-mk_1'}(\Sigma_I^+\times  (\phi_{i_{nm}\ldots i_{(n-k_1)m}})^{-1}B(x, r) )),$$ then using also (\ref{Byr1}) we obtain:
\begin{equation}\label{genericmu2}
\begin{aligned}
\nu_2( (\phi_{i_{nm}\ldots i_{(n-k_1)m}})^{-1}B(x, r) ) &= \hat\mu(\Sigma_I^+\times  (\phi_{i_{nm}\ldots i_{(n-k_1)m}})^{-1}B(x, r) ) = \\ &= \hat\mu(\Phi^{-mk_1'}(\Sigma_I^+\times  (\phi_{i_{nm}\ldots i_{(n-k_1)m}})^{-1}B(x, r) )) \\
& \le 2\mathop{\sum}\limits_{\ul j \ \text{generic}}\hat\mu\big(\tilde A(m, \vp)\cap \big([j_1\ldots j_{mk_1'}] \times \phi_{j_{mk_1'}\ldots j_1}^{-1}  (\phi_{i_{nm}\ldots i_{(n-k_1)m}})^{-1}B(x, r)\big) \big)
\end{aligned}
\end{equation}
But for two generic histories $\ul j, \ul j'$, i.e. for $\ul j, \ul j' \in \pi_1(\tilde A(m, \vp))$, we know that for any 
 \ $(\omega, z) \in \tilde A(m, \vp) \bigcap \big( [j_1\ldots j_{mk_1'}]\times \phi_{j_{mk_1'}\ldots j_1}^{-1} (\phi_{i_{nm}\ldots i_{(n-k_1)m}})^{-1}B(x, r) $ $\bigcup \ [j_1'\ldots j_{mk_1'}']\times \phi_{j_{mk_1'}'\ldots j_1'}^{-1} (\phi_{i_{nm}\ldots i_{(n-k_1)m}})^{-1}B(x, r) \big)$,
$$J_{\Phi^{mk_1'}}(\hat\mu)(\omega, z) \  \in  \ \big(e^{mk_1'(F_\Phi(\hat\mu) -\vp)}, \ e^{mk_1'(F_\Phi(\hat\mu) +\vp)}\big).$$  Hence since $\Phi^{mk_1'}(\tilde A(m, \vp) \ \bigcap \ [j_1\ldots j_{mk_1'}]\times \phi_{j_{mk_1'}\ldots j_1}^{-1} (\phi_{i_{nm}\ldots i_{(n-k_1)m}})^{-1}B(x, r) \subset \phi_{i_{nm}\ldots i_{(n-k_1)m}})^{-1}B(x, r)$ and using the above estimate on the Jacobian of $\hat\mu$ with respect to $\Phi^{mk_1'}$, it follows that    there exists a constant factor $C>1$ such that the ratio of the $\hat\mu$-measures of the two preimage type sets   $$ \tilde A(m, \vp) \ \bigcap \ [j_1\ldots j_{mk_1'}]\times \phi_{j_{mk_1'}\ldots j_1}^{-1} (\phi_{i_{nm}\ldots i_{(n-k_1)m}})^{-1}B(x, r),  \ \text{and}, $$  $$ \tilde A(m, \vp) \ \bigcap \ [j_1'\ldots j'_{mk_1'}]\times \phi_{j'_{mk_1'}\ldots j'_1}^{-1} (\phi_{i_{nm}\ldots i_{(n-k_1)m}})^{-1}B(x, r),$$ corresponding to generic $\ul j, \ul j'$, \ belongs to the interval $\big(C^{-1} e^{-mk_1' \vp}, C e^{mk_1' \vp}\big)$. 
\newline
Notice also that there are at most $d^{mk_1'}$  sets of type $ [j_1\ldots j_{mk_1'}]\times \phi_{j_{mk_1'}\ldots j_1}^{-1} (\phi_{i_{nm}\ldots i_{(n-k_1)m}})^{-1}B(x, r)$, where $d:= |I|$. The maximality assumption for $k_1'$ implies that $\phi_{i_{m(n-k_1)}\ldots i_{m(n-k_1-k_1')}}^{-1}(y) \in X(\alpha; m, \vp)$, where $y:= x_{-mk_1}$. 
Thus using the above discussion and (\ref{genericmu2}), we obtain:
\begin{equation}\label{modphinu}
\nu_2((\phi_{i_{nm}\ldots i_{(n-k_1)m}})^{-1}B(x, r)) \le Cd^{mk_1'(1+\frac{2\vp}{\log d})}\cdot \nu_2(\phi_{i_{m(n-k_1-1)}\ldots i_{m(n-k_1-k_1')}}^{-1} (\phi_{i_{nm}\ldots i_{(n-k_1)m}})^{-1}B(x, r))
\end{equation}

Now from above, $\phi_{i_{m(n-k_1)}\ldots i_{m(n-k_1-k_1')}}^{-1}(y) \in X(\alpha; m, \vp)$, and  we repeat the argument from (\ref{repeatnu2}) along the sequence $\ul i$ until reaching another preimage of $x$ which does not belong to $X(\alpha; m, \vp)$; next, we apply again the argument from (\ref{modphinu}), and so on, until reaching the $nm$-preimage of $x$, namely $x_{-nm}= \phi_{i_{nm}\ldots i_1}^{-1}(x)$. Recall from (\ref{seqx}) that  $x' := x_{-nm}$. Consider now $r_0>0$ a fixed radius, and $n= n(r)$ be chosen so that, due to the conformality of $\phi_i, i \in I$, 
\begin{equation}\label{zeta0}
\phi^{-1}_{i_{nm}\ldots i_1}B(x, r) = B(x_{-nm}, r_0) = B(x', r_0)
\end{equation} 
On the backward $m$-trajectory (\ref{seqx}) of $x$ determined by $\ul i$ above, recall that we denoted by $k_1'$ the length of the first maximum ``gap'' consisting of preimages of $x$ which are not in $X(\alpha; m, \vp)$. Now let us denote in general the lengths of such maximal ``gaps'' in this trajectory (\ref{seqx}), consisting of consecutive preimages which do not belong to $X(\alpha; m, \vp)$, by  $k_1', k_2', \ldots$.  
\ \ More precisely, we have $x_{-jm} \in X(\alpha; m, \vp), 0\le j \le k_1-1$, followed by  $x_{-mk_1}, \ldots, x_{-m(k_1+k_1'-1)} \notin X_m$; then $x_{-m(k_1+k_1')}, \ldots, x_{-m(k_1+k_1'+k_2-1)} \in X_{m}(\alpha, \vp)$, followed by  $x_{-m(k_1+k_1'+k_2)}, \ldots, x_{-m(k_1+k_1'+k_2+k_2'-1)} \notin X(\alpha; m, \vp)$; then,  
\newline
$x_{-m(k_1+k_1'+k_2+k_2')}, \ldots, x_{-m(k_1+k_1'+k_2+k_2'+k_3-1)}$ $\in X(\alpha; m, \vp)$, and so on.  
\newline
Denote by $k_1, \ldots, k_{p(n)}$ and $k_1', \ldots, k_{p'(n)}'$ the integers obtained by this procedure, corresponding to  the sequence (\ref{seqx}). Clearly, $$k_1+k_1' + \ldots + k_{p(n)} + k_{p(n)}' = n$$ Also $p'(n)$ is equal either to $p(n)$ or to $p(n) -1$.
Then from the properties of $\tilde D_m(\alpha, \vp)$ in (\ref{Dam}), 
\begin{equation}\label{gapsbound}
k_1 + \ldots + k_{p(n)} \ge n(1-3\alpha), \ \text{and} \ \  k_1'+k_2'+\ldots + k_{p'(n)}'\le 3\alpha n
\end{equation}
We  apply (\ref{repeatnu2}) for the generic preimages $x, \ldots, x_{-m(k_1-1)}$ from the trajectory (\ref{seqx}), then  apply the estimate in (\ref{modphinu}) for the nongeneric preimages $x_{-mk_1}, \ldots, x_{-m(k_1+k_1'-1)}$, then again apply (\ref{repeatnu2}) for $x_{-m(k_1+k_1')}, \ldots, x_{-m(k_1+k_1'+k_2-1)}$, followed by (\ref{modphinu}) for $x_{-m(k_1+k_1'+k_2)}, \ldots, x_{-m(k_1+k_1'+k_2+k_2'-1)}$, and so on. 
Hence applying succesively the estimates  (\ref{repeatnu2}) and (\ref{modphinu}), and recalling (\ref{zeta0}) and the bound on the combined length of gaps $k_1'+k_2'+\ldots + k_{p'(n)}'\le 3 \alpha n$ from (\ref{gapsbound}), we obtain:
\begin{equation}\label{combest}
\nu_2(B(x, r)) \le (\frac{C}{1-\alpha})^{n}\cdot d^{3\alpha mn(1+\frac{\vp}{\log d})}\cdot e^{(F_\Phi(\hat \mu) - h(\mu^+) +2\vp) \cdot mn(1-3\alpha)} \nu_2(B(x', r_0)).
\end{equation}
But  $\chi_s(\hat\mu)$ is the stable Lyapunov exponent of $\hat\mu$, i.e $\chi_s(\hat\mu) = \int_{\Sigma_I^+\times \Lambda} \log |\phi_{\omega_1}'(x)| \ d\hat\mu(\omega, x)$. From (\ref{zeta0}) it follows  that $r = r_n= r_0|\phi_{i_{nm}\ldots i_1}'(x')|$, and recall that $(\ul i, x') \in \tilde D_m(\alpha, \vp)$. Therefore, 
\begin{equation}\label{rlyap}
e^{nm(\chi_s(\hat\mu)-\vp)} \le r_n \le e^{nm(\chi_s(\hat\mu) +\vp)}.
\end{equation}
Denote by $\tilde X(\alpha; m, \vp):= \pi_2 \tilde D_m(\alpha, \vp)$. Then since $\nu_2:= \pi_{2*}\hat\mu$ and $\hat\mu(\tilde D_m(\alpha, \vp)) > 1-4\alpha$, we obtain,
$$\nu_2(\tilde X(\alpha; m, \vp)) = \hat\mu(\Sigma_I^+\times \tilde X(\alpha; m, \vp)) \ge \hat \mu(\tilde D_m(\alpha, \vp)) \ge 1-4\alpha$$
From (\ref{combest}) and (\ref{rlyap}) it follows that by taking  $C' = 2C>1$, then for every $x \in \tilde X(\alpha; m, \vp)$ and for a sequence $r_n\to 0$, the following estimate holds: 
$$
\nu_2(B(x, r_n)) \le (\frac{C'}{1-\alpha})^{n}\cdot d^{3\alpha mn (1+\frac{2\vp}{\log d})} \cdot r_n^{(1-3\alpha)\cdot \frac{F_\Phi(\hat\mu)-h(\mu^+)+2\vp}{\chi_s(\hat\mu)}}.
$$
Thus,
\begin{equation}\label{Thus14}
\log \nu_2(B(x, r_n)) \le (1-3\alpha) \log r_n\cdot \frac{F_\Phi(\hat\mu)-h(\mu^+)+2\vp}{\chi_s(\hat\mu)} + n \log \frac{C'}{1-\alpha} + 3\alpha mn \cdot \log d (1+\frac{2\vp}{\log d}).
\end{equation}
But by using (\ref{rlyap}) and dividing in (\ref{Thus14}) by $\log r_n$, one obtains:
\begin{equation}\label{nu2r}
\frac{\log \nu_2(B(x, r_n))}{\log r_n} \ge (1-3\alpha) \frac{F_\Phi(\hat\mu)-h(\mu^+)+2\vp}{\chi_s(\hat\mu)} + \frac{\log\frac{C'}{1-\alpha}}{m(\chi_s(\hat\mu)+\vp)} + (1+\frac{2\vp}{\log d}) \frac{3 \alpha \log d}{\chi_s(\hat\mu) + \vp}.
\end{equation}
Now let us take some arbitrary radius $\rho>0$, and assume that for some integer $n$, we have $r_{n+1}\le \rho \le r_n$, where $r_n$ is defined at (\ref{rlyap}). Then $\nu_2(B(x, \rho)) \le \nu_2(B(x, r_n))$, therefore $$\frac{\log\nu_2(B(x, \rho))}{\log \rho} \ge \frac{\log \nu_2(B(x, r_n))}{\log \rho}$$ But $\log \rho \ge \log r_{n+1}$, hence $\frac{1}{\log \rho} \le \frac{1}{\log r_{n+1}} < 0$, hence from above,
\begin{equation}\label{rhon}
\frac{\log\nu_2(B(x, \rho))}{\log \rho} \ge \frac{\log\nu_2(B(x, r_n))}{\log r_{n+1}} \ge \frac{\log \nu_2(B(x, r_n))}{c+\log r_n},
\end{equation}
since $r_{n+1} > cr_n$, for some constant $c$ independent of $n$.
Thus by letting $\rho \to 0$, and using (\ref{nu2r}) and (\ref{rhon}), we obtain the following lower estimate for the lower pointwise dimension of $\nu_2$:
$$\underline{\delta}(\nu_2)(x) \ge (1-3\alpha) \frac{F_\Phi(\hat\mu)-h(\mu^+)+2\vp}{\chi_s(\hat\mu)} + \frac{\log\frac{C'}{1-\alpha}}{m(\chi_s(\hat\mu)+\vp)} + 3\alpha(1+\frac{2\vp}{\log d}) \frac{\log d}{\chi_s(\hat\mu) + \vp}
$$
But  $\alpha \to 0$ as $m \to \infty$, and $\nu_2(\tilde X(\alpha; m, \vp)) \to 1$. 
So from last displayed estimate, for $\nu_2$-a.e $x \in \Lambda$, 
$$\underline{\delta}(\nu_2)(x) \ge  \frac{F_\Phi(\hat\mu)-h(\mu^+)+2\vp}{\chi_s(\hat\mu)} $$
Since $\vp$ is arbitrary, it follows from the above lower estimate and  (\ref{upperestd}) that, for $\nu_2$-a.e $x \in \Lambda$, 
$$\delta(\nu_2)(x) =  \frac{F_\Phi(\hat\mu)-h(\mu^+)}{\chi_s(\hat\mu)}$$
From (\ref{oS}), $F_\Phi(\hat\mu)) = \log o(\mathcal S, \hat \mu)$, so using the last formula we conclude the proof of Theorem \ref{thm1}.

$\hfill\square$



\textit{Proof of  Theorem \ref{bernou}.} 

From the Birkhoff Ergodic Theorem applied to the continuous potential $\kappa(\omega, x) = \log|\phi'_{\omega_1}(x)|$ on $\Sigma_I^+ \times \Lambda$, we see that $\chi_s(\hat\mu_{\bf p}) = \mathop{\lim}\limits_{n \to \infty} \frac{\log|\phi'_{\omega_n\ldots\omega_1}(x)|}{n}$, for $\hat\mu_{\bf p}$-a.e $(\omega, x)$. But from the Bounded Distortion Property, it does not matter which $x$ we take in the limit above. Hence using that $\mu_{\bf p } = \pi_{1*}\hat\mu_{\bf p}$ and that $\mu_{\bf p}$ is a Bernoulli measure, it follows that 
 $$
 \begin{aligned}
 \chi_s(\hat\mu_{\bf p}) =& \int_{\Sigma_I^+\times \Lambda} \kappa(\omega, x) d\hat\mu_{\bf p} = \int_{\Sigma_I^+\times \Lambda}  \log|\phi'_{\omega_1}|(\pi\sigma\omega) d\hat\mu_{\bf p} \\ &= \int_{\Sigma_I^+} \log|\phi'_{\omega_1}|(\pi\sigma\omega) d\mu_{\bf p} = \chi(\mu_{\bf p})\end{aligned}$$ Thus from Theorem \ref{thm1} we obtain the dimension formula. 

$\hfill\square$


\textit{Proof of Corollary \ref{qint}.}

First  recall the definition of $\beta_n(x)$ from (\ref{beta}) as being the number of multi-indices $(i_1, \ldots, i_n)\in \{1, \ldots, m\}^n$ so that $x \in \phi_{i_1}\circ\ldots \circ\phi_{i_n}(\Lambda)$. Recall that $U_j$ is the union of cylinders from the set $G_j$ in (\ref{gj}), for $1 \le j \le s$.
If  $x = \pi(\omega)$ and the maps $\phi_{j_1\ldots j_q}$  are grouped as in the statement, and if $\sigma^{q\ell}\omega \in U_{k_\ell}$ for $1 \le \ell \le n$, then using the notation in (\ref{mj}) we have the following bound on $\beta_{qn}(x)$,
\begin{equation}\label{betanx}
\beta_{qn}(x) \le m_{k_1}\ldots m_{k_n}
\end{equation}

Recall now that $\theta:\Sigma_m^+ \to \mathbb R$, $\theta(\omega) = \log m_j$, when $\omega \in U_j$ for some $1 \le i \le s$. From the definition, one sees that $\theta$ is a H\"older continuous potential on $\Sigma_m^+$. Also from (\ref{betanx}) we have,  $$\log \beta_{qn}(\pi\omega) \le S_{n, q}(\theta)(\omega),$$ where $S_{n, q}(\theta)(\omega)$ is the consecutive sum of $\theta$ with respect to  $\sigma^q$. Without loss of generality, assume $q =1$. 
From the definition of $b_n((\omega, x), \tau, \hat\mu_\psi)$, it follows that for any $n \ge 1, \tau>0, (\omega, x) \in \Sigma_m^+\times \Lambda$,  $$b_n((\omega, x), \tau, \hat\mu_\psi) \le  \beta_n(\phi_{\omega_n\ldots\omega_1}(x)) = \beta_n\circ \pi_2  \circ \Phi^n(\omega, x) $$
Hence,
\begin{equation}\label{intineq}
\begin{aligned}
\int_{\Sigma_m^+\times \Lambda} \log b_n((\omega, x), \tau, \hat\mu_\psi) \ d\hat\mu_\psi (\omega, x) \le & \int_{\Sigma_m^+\times \Lambda} \log \beta_n\circ \pi_2 \circ \Phi^n(\omega, x) d\hat\mu_\psi(\omega, x) \\ = & \int_{\Sigma_m^+\times \Lambda} \log \beta_n\circ \pi_2(\omega, x) d\hat\mu_\psi.
\end{aligned}
\end{equation}
Now take the lift of $\Phi$, namely $$\tilde \Phi: \Sigma_m^+\times \Sigma_m^+ \to \Sigma_m^+\times \Sigma_m^+, \ \tilde \Phi(\omega, \eta) = (\sigma \omega, \omega_1\eta)$$ 
If $\tilde \mu_\psi$ is the equilibrium state of $\psi\circ \pi_1:\Sigma_m^+\times \Sigma_m^+ \to \mathbb R$ for $\tilde \Phi$, then $\pi_{1*}\tilde \mu_\psi = \hat\mu_\psi$. 
Let $\tilde \pi_2:\Sigma_m^+\times \Sigma_m^+\to \Sigma_m^+$ be the  projection to  second coordinate. Then for any $(\omega, \eta) \in \Sigma_m^+\times \Sigma_m^+$, one has $\beta_n\circ \pi(\eta) = \beta_n\circ \pi_2(\omega, \pi(\eta))$. Thus from above and continuing (\ref{intineq}),  with $q$ assumed to be 1,
\begin{equation}\label{intineq2}
 \int_{\Sigma_m^+\times \Lambda} \log \beta_n\circ \pi_2(\omega, x) \ d\hat\mu_\psi(\omega, x) = \int_{\Sigma_m^+\times \Sigma_m^+} \log \beta_n\circ \pi(\eta)  d\tilde \mu_\psi(\omega, \eta) \le \int_{\Sigma_m^+}S_n\theta(\eta) \ d\tilde \pi_{2, *}\tilde \mu_\psi(\eta)
 \end{equation}
 
Now let us see more closely what is the measure $\tilde \pi_{2, *}\tilde \mu_\psi$. For a cylinder $[i_1\ldots i_n] \subset \Sigma_m^+$, we know that $\tilde \Phi^n([i_1\ldots i_n]\times \Sigma_m^+) = \Sigma_m^+\times [i_n\ldots i_1]$, so by the $\tilde \Phi$-invariance of $\tilde \mu_\psi$ and since $\pi_{1, *}\tilde \mu_\psi = \mu_\psi$, it follows that $\tilde\pi_{2, *}\tilde \mu_\psi([i_n\ldots i_1]) = \tilde\mu_\psi([i_1\ldots i_n]\times \Sigma_m^+) =   \mu_\psi([i_1\ldots i_n])$. 
 Now as $\theta$ is constant on 1-cylinders (as $q=1$), it implies that when we compute the integral over $\Sigma_m^+$ (thus considering all the cylinders), we obtain:
$$\int_{\Sigma_m^+}S_n\theta(\omega) d\tilde \pi_{2, *}\tilde \mu_\psi (\omega)  = \int_{\Sigma_m^+} S_n\theta(\omega) d\mu_\psi(\omega) = n\int_{\Sigma_m^+}\theta(\omega)d\mu_\psi(\omega)$$
Therefore, from (\ref{intineq}), (\ref{intineq2}) and the definition of the overlap number $o(\mathcal S, \hat\mu_\psi)$, we infer that:
$$\log o(\mathcal S, \hat\mu_\psi) \le  \int_{\Sigma_m^+} \theta(\omega)  d\mu_\psi (\omega).$$
In the general case, for arbitrary $q\ge 1$, we obtain similarly,
$$
\begin{aligned}
&\mathop{\lim}\limits_{n \to \infty}\frac {1}{qn} \int_{\Sigma_m^+\times \Lambda} \log \beta_{qn}(\pi\omega) d\hat\mu_\psi(\omega, y) \le \mathop{\lim}\limits_{n \to \infty} \frac {1}{qn} \int_{\Sigma_m^+} S_{n, q}(\theta) (\omega)d\mu_\psi(\omega) = \frac {1}{q}\int_{\Sigma_m^+} \theta(\omega) d\mu_\psi (\omega).
\end{aligned}
$$
Therefore it follows that,   $$\log o(\mathcal S, \hat\mu_\psi) \le \exp(\mathop{\lim}\limits_{n \to \infty}\frac {1}{n} \int_{\Sigma_m^+\times \Lambda} \log \beta_n(\pi\omega) d\hat\mu_\psi(\omega, y)) \le \frac {1}{q}\int_{\Sigma_m^+} \theta(\omega) d\mu_\psi (\omega).$$
Hence from Theorem \ref{thm1} and the last displayed inequality we obtain,
$$HD(\pi_{2*}\hat\mu_\psi) \ge \frac{h(\mu_\psi) - \frac {1}{q}\int_{\Sigma_m^+} \theta(\omega) \ d\mu_\psi (\omega)}{|\chi_s(\hat\mu_\psi)|}. $$ $\hfill\square$

\textbf{Acknowledgements:} This work was supported by  grant PN-III-P4-ID-PCE-2020-2693 
from Ministry of Research and Innovation, CNCS/CCCDI - UEFISCDI Romania. The author also thanks Yakov Pesin for discussions
during a visit at Penn State University.

 Eugen Mihailescu, 
Institute of Mathematics of the Romanian Academy, 
C. Grivitei 21,  
Bucharest, Romania. \ \ \ \ \ \ 

Email:   Eugen.Mihailescu\@@imar.ro
\ \ \ \ \ \
Web:  www.imar.ro/$\sim$mihailes

\end{document}